\documentclass[12pt]{article}
\usepackage[centertags]{amsmath}
\usepackage{amsfonts}
\usepackage{amssymb}
\usepackage{amsthm}
\usepackage{newlfont}
\usepackage[left=1in,top=1in,right=1in,nohead,foot=0.5in]{geometry}
\usepackage[all]{xy}
 \newtheorem{thm}{Theorem}[section]
 \newtheorem{cor}[thm]{Corollary}
 \newtheorem{lem}[thm]{Lemma}
 \newtheorem{prop}[thm]{Proposition}
 \newtheorem{conj}{Conjecture}[section]

 \theoremstyle{definition}
 \newtheorem{defn}[thm]{Definition}
 \theoremstyle{remark}
 \newtheorem{rem}[thm]{Remark}
 \numberwithin{equation}{section}

 \newcommand{\GL}{\operatorname{GL}}

 \newcommand{\chr}{\operatorname{char}}

 \newcommand{\Vol}{\operatorname{Vol}}

\newlength{\defbaselineskip}
\begin{document}

\title{Isolated Periodic Points in Several Nonarchimedean Variables}
\author{Alon Levy}

\maketitle

\abstract{Let $\varphi: \mathbb{P}^{n}_{F} \to \mathbb{P}^{n}_{F}$ where $F$ is a complete valued field. If $x$ is a fixed point, such that the action of $\varphi$ on $T_{x}$ has eigenvalues $\lambda_{1}, \ldots, \lambda_{n}$, with $\lambda_{1}, \ldots, \lambda_{r}$ not contained in the multiplicative group generated by $\lambda_{r+1}, \ldots, \lambda_{n}$, then $\varphi$ has a codimension-$r$ fixed formal subvariety. Under mild assumptions, this subvariety is analytic. We use this to prove two results. First, we generalize results of Rivera-Letelier on isolated periodic points to higher dimension: if $F$ is $p$-adic, and each $|\lambda_{i}| \leq 1$, then there is an analytic neighborhood of $x$ without any other periodic points. And second, we prove Zhang's conjecture that there exists a $\overline{\mathbb{Q}}$-point with Zariski-dense forward orbit in two cases, extending results of Amerik, Bogomolov, and Rovinsky.}

\section{Introduction}

In the theory of dynamical systems on $\mathbb{P}^{n}$ over some complete valued field $F$, one of the tools for analyzing dynamical behavior is the linearization:

\begin{defn}\label{linearizable1}Let $\varphi: \mathbb{P}^{n}_{F} \to \mathbb{P}^{n}_{F}$, and suppose $\varphi(x) = x$. Near $x$, $\varphi$ is given by $n$ power series in $n$ variables, which we write as $\varphi_{1}(x_{1}, \ldots, x_{n}), \ldots, \varphi_{n}(x_{1}, \ldots, x_{n})$. We say that $\varphi$ is \textbf{formally linearizable} at $x$ if there exists power series $g_{1}, \ldots, g_{n} \in F[[x_{1}, \ldots, x_{n}]]$ and a matrix $A$ such that $\varphi_{i}(g_{1}, \ldots, g_{n}) = g_{i}(A\mathbf{x})$ for each $i$; if $\varphi$ is linearizable, $A$ must be the action of $\varphi$ on $T_{x}$. Moreover, if each $g_{i}$ converges in some polydisk around $x$, we say $\varphi$ is \textbf{locally linearizable} and the linearization is local-analytic.\end{defn}

Whether $\varphi$ is linearizable depends intimately on the eigenvalues of $A = \varphi_{*}T_{x}$, which we call \textbf{multipliers}. In $1$ dimension, it is conventional to write the multiplier as $\lambda$; at a fixed point $x \in \mathbb{A}^{1}$, it is equal to $\varphi'(x)$. The best results on linearization in $1$ dimension are collected below:

\begin{thm}\label{infodump}\cite{RL, Yoc} Suppose $\varphi: \mathbb{P}^{1} \to \mathbb{P}^{1}$ and $\varphi(0) = 0$. For any $F$, $\varphi$ is formally linearizable at $0$ if and only if $\lambda = \varphi'(0)$ is not zero or a root of unity. If, in the topology of $F$, $0 < |\lambda| \neq 1$, then $\varphi$ is also locally linearizable. Now, assume $\chr F = 0$. If $\lambda = 0$ then there is an analytic $g$ such that $\varphi(g(z)) = g(z^{e})$ for some integer $e > 1$, and if $\lambda = 1$ then there is an analytic $g$ such that $\varphi(g(z)) = g(z + z^{e})$. If $|\lambda| = 1$ and $\lambda$ is not a root of unity, then whether $\varphi$ is linearizable depends on the topology on $F$: when $F$ is nonarchimedean and $\chr F = 0$, $\varphi$ is always linearizable; and when $F = \mathbb{C}$, it may or may not be linearizable, depending on the irrational $\theta = (\log \lambda)/2\pi i$ -- if $\theta$ is very poorly approximated by rationals, e.g. if it is algebraic, then $\varphi$ is linearizable, and if $\theta$ is very well approximated, e.g. if it is a Liouville number, then $\varphi$ is not linearizable.\end{thm}

In higher dimension, Herman-Yoccoz prove in~\cite{HY} that,

\begin{thm}\label{hy}Suppose $\varphi: \mathbb{P}^{n} \to \mathbb{P}^{n}$ and $\varphi(x) = x$. If $x$ has distinct, multiplicatively independent multipliers, then $\varphi$ is formally linearizable at $x$. If moreover $F$ is nonarchimedean, $\chr F = 0$, and the multipliers are algebraic numbers, then $\varphi$ is locally linearizable.\end{thm}

The linearization, and the associated analytic coordinate changes to $z^{e}$ and $z + z^{e}$ when linearization is not possible, is a key tool in studying local dynamics; see~\cite{D1C, RL} for more background in $1$ dimension. In~\cite{ABR}, Amerik-Bogomolov-Rovinsky apply the linearization tool of~\cite{HY} to questions of local behavior, with an eye toward proving results in algebraic geometry. One particular application of interest is the question of isolation:

\begin{defn}Let $X$ be a smooth projective variety over a complete valued field $F$, let $\varphi: X \to X$ be a dynamical system defined over $F$, and let $x$ be a fixed point. We say $x$ is \textbf{isolated} if there exists a neighborhood of $x$ containing no periodic points other than $x$.\end{defn}

This question has been extensively studied in $1$ dimension. The best results are below:

\begin{thm}\label{infodumpsio}\cite{Yoc, RL, Bez, LRL, LRL2} Let $\varphi: \mathbb{P}^{1}_{F} \to \mathbb{P}^{1}_{F}$. If $F = \mathbb{C}$ then $0$ is isolated if and only if $|\lambda| < 1$, or $|\lambda| = 1$ and $\varphi$ is locally linearizable. If $\chr F = 0$ and $F$ is nonarchimedean, then $0$ is isolated if and only if $|\lambda| \leq 1$. If $\chr F > 0$, then $0$ is isolated if $|\lambda| < 1$, or if $|\lambda| = 1$ and $\lambda$ is not a root of unity; if $\lambda$ is a root of unity, then for generic $\varphi$, $0$ is isolated; if $|\lambda| > 1$ and $\chr F > \deg\varphi$ then $0$ is not isolated.\end{thm}

It is conjectured in~\cite{LRL, LRL2} that if $\chr F > 0$ and $\lambda$ is a root of unity, then $0$ is still isolated.

In this paper, we study local dynamics in several variables. While the question of linearization has already been investigated in several variables, the question of whether periodic points are isolated remains unsolved. In this paper, we prove,

\begin{thm}\label{isolatedper}Let $\varphi: \mathbb{P}^{n} \to \mathbb{P}^{n}$ be an algebraic morphism defined over a complete nonarchimedean field $F$ of characteristic $0$, and let $x \in \mathbb{P}^{n}(\overline{F})$ be a fixed point under $\varphi$. If the multipliers of $x$ are all attracting or indifferent, then $x$ is isolated as a periodic point -- that is, there exists a neighborhood of $x$ in the analytic topology on $F$ containing no periodic points except $x$.\end{thm}

Theorem~\ref{isolatedper} is a step toward proving a conjecture of Zhang~\cite{Zha}:

\begin{conj}\label{introzhang}(Zhang) Let $\varphi: \mathbb{P}^{n} \to \mathbb{P}^{n}$ be a morphism defined over a number field $K$. Then there exists a point $z \in \mathbb{P}^{n}(\overline{K})$ whose forward orbit under $\varphi$ is Zariski-dense.\end{conj}

In brief, to prove Conjecture~\ref{introzhang}, we could prove that a fixed point $x$ is isolated not just as a periodic point, but also as a point with non-dense Zariski orbit. Specifically, in many cases, we can show that there exists an analytic neighborhood of a fixed point $x$ in which all points outside a finite union of hypersurfaces have forward orbits that cannot be contained in any reasonable analytic subvariety, including all the ones arising from local analytic deformations of algebraic subvarieties. In~\cite{ABR}, Conjecture~\ref{introzhang} is proven in case the multipliers of $x$ are multiplicatively independent.

If the multipliers are not multiplicatively independent, then it is instead possible that all points near a fixed point fall into fixed analytic subvarieties. As an example, if $\varphi(x_{0}:\ldots:x_{n}) = (x_{0}^{d}:\ldots:x_{n}^{d})$, then near the fixed point $x = (1:\ldots:1)$, any subvariety defined by monomial equations, of the form $x_{0}^{\alpha_{0}}\ldots x_{n}^{\alpha_{n}} = 1$, is fixed; such a subvariety is algebraic if all exponents $\alpha_i$ are rational, but can also be defined analytically near $x$ for any $(\alpha_{0}:\ldots:\alpha_{n}) \in \mathbb{P}^{n}_{F}$. In order to find a $\overline{\mathbb{Q}}$-rational point with Zariski-dense image, we rely on the fact that the linearization of the power map, the logarithm, doesn't map any algebraic numbers to algebraic numbers (except sending $1$ to $0$), and if we pick multiplicatively independent algebraics $x_{0}, \ldots, x_{n}$, they will only satisfy the condition $x_{0}^{\alpha_{0}}\ldots x_{n}^{\alpha_{n}} = 1$ for transcendental $\alpha_{i}$. To generalize this argument to other morphisms with fixed points with multiplicatively dependent multipliers requires generalizing the various theorems on the independence of logarithms to dynamical linearization maps.

However, in some cases of multiplicatively dependent multipliers, we are capable of proving Conjecture~\ref{introzhang} anyway. In fact, we show the following:

\begin{thm}\label{zhangcases} Let $\varphi: \mathbb{P}^{n} \to \mathbb{P}^{n}$ be a morphism defined over a number field $K$. Then there exists a point $z \in \mathbb{P}^{n}(\overline{K})$ whose forward orbit under $\varphi$ is Zariski-dense, provided one of two sufficient conditions is met:

\begin{enumerate}
\item There exists a periodic point $x$ one of whose multipliers is equal to $0$ and the rest are multiplicatively independent.
\item $n = 2$, and there exists a periodic point $x$ exactly one of whose multipliers is a root of unity.
\end{enumerate}\end{thm}

We prove both Theorem~\ref{isolatedper} and Theorem~\ref{zhangcases} by using local analytic variable changes to do something similar to a linearization. We cannot quite linearize $\varphi$, since it is not always possible when the multipliers are multiplicatively dependent, and definitely not possible when some of the multipliers are zero or roots of unity. But we can still perform local analytic changes that replace $\varphi$ with a simpler map. In particular, we can separate many groupings of multipliers and treat each separately: we can separate $0$ from nonzero multipliers, or separate \textbf{repelling} multipliers (those with $|\lambda_{i}| > 1$) from nonrepelling multipliers, or separate \textbf{attracting} multipliers (those with $|\lambda_{i}| < 1$) from nonattracting ones. This generalizes well-known results in hyperbolic dynamics in several complex variables on separating attracting and repelling directions; see for example~\cite{Yoc2}. Among the \textbf{indifferent} multipliers (those with $|\lambda_{i}| = 1$), we can formally separate the \textbf{irrationally indifferent} ones (those that are not roots of unity) from the \textbf{rationally indifferent} ones (those that are roots of unity), but the formal conjugacy may fail to converge in any open neighborhood of $x$.

We describe this method of separating directions in detail in Section~\ref{separating}. We show when we can separate directions formally, and also discuss sufficient conditions to obtain a positive radius of convergence, in both the complex topology and nonarchimedean topologies. This is encoded in Lemmas~\ref{formsep} and~\ref{anal1}.

In Section~\ref{isolated}, we use the machinery we develop in Section~\ref{separating} to prove Theorem~\ref{isolatedper}. The theorem is trivial if all multipliers are attracting, so we separate attracting and indifferent directions, and prove the result for attracting directios only, which is computationally easier. Here, a critical step relies on the so-called \textbf{Newton polytope}, a generalization of the one-dimensional Newton polygon developed in some detail in the work of Rabinoff (\cite{Rab}). This generalizes one-dimensional results of Rivera-Letelier~\cite{RL}, and provides a framework that could be used to prove a generalization of Theorem~\ref{isolatedper} in positive characteristic, which is conjectured also in one dimension but only proven in the generic case and not in all cases (see~\cite{LRL, LRL2}).

Finally, in Section~\ref{zhang}, we prove Theorem~\ref{zhangcases}. Proving Theorem~\ref{zhangcases} in case the first condition is met is a straightforward application of the concept of separating attracting and indifferent directions, where we use the $v$-adic topology where $v$ is any nonarchimedean place of $K$ such that the nonzero multipliers of the periodic point $x$ are all indifferent. The second condition is more delicate: in the worst case, we may have to work in the complex topology, and locally conjugate $\varphi = (\varphi_{1}, \varphi_{2})$ to the form $\varphi_{2}(x_{1}, x_{2}) = x_{2} + x_{2}^{e}$. Since we need to work with complex variables and not just nonarchimedean variables, we develop the method of Section~\ref{separating} in the complex case and not only in the easier nonarchimedean case.

We do not use any of the results from Section~\ref{isolated} in Section~\ref{zhang}; we only use one definition, Definition~\ref{minpoly}, in one part of the proof of the second case of Theorem~\ref{zhangcases}. Theorems~\ref{isolatedper} and~\ref{zhangcases} are two separate applications of Lemmas~\ref{formsep} and~\ref{anal1}. Thus, the reader may read Sections~\ref{isolated} and~\ref{zhang} in either order.

\textbf{Acknowledgements.} The author would like to thank Shouwu Zhang and Tom Tucker for reading a draft of the paper and for helpful comments. The author would also like to thank P\"{a}r Kurlberg for many productive discussions in the early stages of this project. This research was funded by the G\"{o}ran Gustafsson Foundation.

\section{Separating Directions}\label{separating}

From this section on, we will work over an arbitrary field $F$, in dehomogenized coordinates, unless stated otherwise. We will dehomogenize with respect to $x_{0}$, in order to retain the notation $x_{1}, \ldots, x_{n}$. Therefore, if $\psi: \mathbb{P}^{n} \to \mathbb{P}^{n}$ is in homogeneous coordinates $\psi = (\psi_{0}:\ldots:\psi_{n})$, we write $\varphi = (\varphi_{1}, \ldots, \varphi_{n})$ with $\varphi_{i} = \psi_{i}/\psi_{0}$. The coordinate maps $\varphi_{i}$ are rational maps in the field $F(x_{1}, \ldots, x_{n})$. We still assume that $\psi$ is a morphism from $\mathbb{P}^{n}$ to itself; in terms of $\varphi$, this means that the numerators of the $\varphi_{i}$s and the common denominator $\psi_{0}$ do not share a nontrivial common zero. This does not mean $\varphi$ is a morphism from $\mathbb{A}^{n}$ to itself -- in general, it is only a rational map -- but it does mean that it extends to a morphism from $\mathbb{P}^{n}$ to itself.

We assume that $x$ is a fixed point and not just a periodic point, replacing $\varphi$ with an iterate if possible. We also assume that $x = (0, \ldots, 0)$ and that the rational canonical form of the action $\varphi^{*}(T_{x})$ is a Jordan canonical form, replacing $F$ by a finite extension if necessary. In both this section and Section~\ref{isolated} we can assume $F$ is algebraically closed, but in Section~\ref{zhang}, we require $F$ to be locally compact, and therefore we do not assume $F$ is algebraically closed here. Our first result is that we can write the rational functions $\varphi_{i}$ as power series:

\begin{prop}\label{comeon}Let $\psi: \mathbb{P}^{n} \to \mathbb{P}^{n}$, labeling $\psi$ as $(\psi_{0}, \ldots, \psi_{n})$. Suppose $x = (1 : 0 : \ldots : 0)$ is a fixed point. If we dehomogenize by taking $x_{0} = 1$, then $1/\psi_{0}$ is a power series in $x_{1}, \ldots, x_{n}$, and therefore, so is $\varphi_{i} = \psi_{i}/\psi_{0}$. If $F$ has any topology, then these power series all converge in some open neighborhood of $x = (0, \ldots, 0)$, which may be taken to be a polydisk centered at $x$ missing the zero-set of $\psi_{0}$.\end{prop}

\begin{proof}This is well-known, but we provide a few details. Since $\psi_{i}(x) = 0$ for all $i > 0$, we must have $\psi_{0}(x) \neq 0$. Thus, $1/\psi_{0}$ has a power series expansion around $0$; this is what it means for $F[[x_{1}, \ldots, x_{n}]]$ to be a local ring, with maximal ideal $(x_{1}, \ldots, x_{n})$. Moreover, suppose without loss of generality that $\psi_{0}$ contains $x_{1}, \ldots, x_{r}$-terms. Then, first, $x_{r+1}, \ldots, x_{n}$ are free. Now, fix $x_{2}, \ldots, x_{r}$ to lie in any polydisk $D$ small enough that the $x_{1}$-free terms of $\psi_{0}$ never evaluate to $0$. For each choice of $x_{2}, \ldots, x_{r}$, we obtain a power series in $x_{1}$, whose radius of convergence equals the distance from $0$ to the nearest root. Choose the minimal radius as $(x_{2}, \ldots, x_{n})$ range over $D$, and observe that it is positive. This gives us a polydisk of convergence.\end{proof}

\begin{rem}If $K$ is a global field, then for all but finitely many valuations $v$ of $K$, setting $F = K_{v}$ gives convergence in the open unit polydisk.\end{rem}

In the sequel, we will assume $\varphi = (\varphi_{1}, \ldots, \varphi_{n})$ is in fact a list of power series in $F[[x_{1}, \ldots, x_{n}]]$. Observe that every $\varphi_{i}$ has zero constant term, since $x = (0, \ldots, 0)$ is fixed. If $F$ has a topology, then we apply linear coordinate change to ensure that all terms of all power series, except possibly the linear terms, have absolute value at most $1$; this forces the power series to converge in the open unit polydisk. Finally, we apply linear coordinate change so that the matrix whose $(i, j)$-term is the $x_{j}$-coefficient of $\varphi_{i}$ is in Jordan canonical form; we write $\lambda_{i}$ for the multiplier in the $x_{i}$-direction.

Our tool in this paper is separating groups of tangent directions leading out of $x$, such as attracting and repelling directions, or attracting and indifferent directions. This is encoded in the following result:

\begin{lem}\label{formsep}Suppose $\lambda_{1}, \lambda_{2}, \ldots, \lambda_{r}$ do not lie in the multiplicative semigroup generated by $\lambda_{r+1}, \ldots, \lambda_{n}$. Then there exists a unique formal analytic subvariety $V$ defined by equations $x_{i} = f_{i}(x_{r+1}, \ldots, x_{n})$ for each $i$ such that $V$ is fixed under the action of $\varphi$. $V$ defines a formal conjugacy $x_{i} \mapsto x_{i} - f_{i}$ to a formal map under which the linear subvariety $x_{1} = \ldots = x_{r} = 0$ is fixed, or, in other words, power series $\varphi_{1}, \ldots, \varphi_{n}$ such that if $x_{1} = \ldots = x_{r} = 0$ then $\varphi_{1} = \ldots = \varphi_{r} = 0$.\end{lem}

\begin{proof}We will construct the defining equations $f_{i}$ for $V$ explicitly. First, we denote the $x_{r+1}^{\alpha_{r+1}}\ldots x_{n}^{\alpha_{n}}$-coefficient of $f_{i}$ by $c_{i, \mathbf{\alpha}}$; we will freely use the notation $\textbf{e}_{j}$ for the vector $\mathbf{\alpha}$ such that $\alpha_{j} = 1$ and all other $\alpha_{k}$s are zero. Now, from the definition of $V$, the lowest-degree terms of $f_{i}$ are quadratic (although in fact if we only assume $V$ is of the form $x_{i} = f_{i}$ with $f_{i}$ having zero constant, we could easily see that $f_{i}$ must have zero linear terms). Let us construct the quadratic terms of $f_{i}$ by using the fact that $V$ is fixed.

We have $\varphi_{i} = f_{i}(\varphi_{r+1}, \ldots, \varphi_{n})$. Assume for now that $x_{i}$ is an eigenvector of $\varphi$ -- that is, it is located at the right, or the bottom, of its Jordan block. Thus, the left-hand side can be written as $\lambda_{i}x_{i} + \mbox{(higher-order terms)}$. Note that these higher-order terms could be quadratic in $F[[x_{r+1}, \ldots, x_{n}]]$, if for example $\varphi_{i}$ contains such terms as $x_{r+1}^{2}$ or $x_{r+1}x_{n}$, but they will not change our computation, as those terms do not involve any $c_{i, \mathbf{\alpha}}$. Now, to find the quadratic terms of $f_{i}(\varphi_{r+1}, \ldots, \varphi_{n})$, we note that the lowest-degree terms of $f_{i}$ are quadratic, and therefore we do not need to plug in the full power series $\varphi_{j}$ into $f_{i}$, but only their linear terms, which are $\lambda_{j}x_{j}$ or $\lambda_{j}x_{j} + x_{j+1}$.

We now work Jordan block by Jordan block, within the group $x_{r+1}, \ldots, x_{n}$. Consider just the first Jordan block, say of size $s$, from $x_{r+1}$ to $x_{r+s}$. The quadratic terms of $f_{i}(\varphi_{r+1}, \ldots, \varphi_{n})$ involving elements of this Jordan block are $$c_{i, \mathbf{e}_{r+j} + \mathbf{e}_{r+k}}(\lambda_{r+j}x_{r+j} + (1-\delta_{js})x_{r+j+1})(\lambda_{r+k}x_{r+k} + (1-\delta_{ks})x_{r+k+1})$$ where $\delta_{js}$ is the Kronecker delta and $j$ and $k$ may be the same or different indices. From this expression we subtract $\lambda_{i}c_{i, \mathbf{e}_{r+j} + \mathbf{e}_{r+k}}x_{r+j}x_{r+k}$; to obtain a unique set of quadratic terms of $f_{i}$, we need to show that the resulting system of linear coefficients of $x_{r+j}x_{r+k}$ is nonsingular. (Other quadratic terms, such as those coming from $\varphi_{i}$, would be the constants in this system of linear equations.) But now, equating $x_{r+1}^{2}$-terms, we obtain the term $c_{i, 2\mathbf{e}_{r+1}}(\lambda_{r+1}^{2} - \lambda_{i})$. Since $\lambda_{i}$ is not in the multiplicative semigroup generated by $\lambda_{r+1}, \ldots, \lambda_{n}$, this is nonzero, and we get a unique $c_{i, 2\mathbf{e}_{r+1}}$.

We apply induction at the first of three levels now. For any $c_{i, \mathbf{e}_{r+j} + \mathbf{e}_{r+k}}$, we obtain the term $c_{i, \mathbf{e}_{r+j} + \mathbf{e}_{r+k}}(\lambda_{r+j}\lambda_{r+k} - \lambda_{i})$ plus terms involving $c_{i, \mathbf{e}_{r+j-1} + \mathbf{e}_{r+k}}$, $c_{i, \mathbf{e}_{r+j} + \mathbf{e}_{r+k-1}}$, and $c_{i, \mathbf{e}_{r+j-1} + \mathbf{e}_{r+k-1}}$. Once we have obtained $c_{i, 2\mathbf{e}_{r+1}}$, we get $2c_{i, 2\mathbf{e}_{r+1}}x_{r+1}x_{r+2} + \lambda_{r+1}\lambda_{r+2}c_{i, \mathbf{e}_{r+1} + \mathbf{e}_{r+2}} - \lambda_{i}c_{i, \mathbf{e}_{r+1} + \mathbf{e}_{r+2}}$ plus some terms not involving any of the coefficients of $f_{i}$. We now have a unique $c_{i, \mathbf{e}_{r+1} + \mathbf{e}_{r+2}}$ since $\lambda_{r+1}\lambda_{r+2} - \lambda_{i} \neq 0$; then by induction we can always replace $r+j$ with $r+j+1$ or $r+k$ with $r+k+1$ and obtain a unique coefficient. This shows that the quadratic coefficients within each Jordan block of $f_{i}$ are unique.

The same argument also works for coefficients that mix different Jordan blocks. Consider the second Jordan block, say of size $t$, from $x_{r+s+1}$ to $x_{r+s+t}$. We can equate $x_{r+1}x_{r+s+1}$-terms and obtain a unique $c_{i, \mathbf{e}_{r+1} + \mathbf{e}_{r+s+1}}$, and apply induction, again with each $c_{i, \mathbf{e}_{r+j} + \mathbf{e}_{r+k}}$ uniquely determining $c_{i, \mathbf{e}_{r+j+1} + \mathbf{e}_{r+k}}$ and $c_{i, \mathbf{e}_{r+j} + \mathbf{e}_{r+k+1}}$. This together gives us unique quadratic coefficients for $f_{i}$.

We now apply a second induction: suppose that $x_{i}$ is not at the end of its Jordan block, but that it falls into a Jordan block $x_{1}, \ldots, x_{u}$, and we have already uniquely determined the quadratic terms of $f_{i+1}, \ldots, f_{u}$. Then the left-hand side of $\varphi_{i} = f_{i}(\varphi_{r+1}, \ldots, \varphi_{n})$ can be written as $\lambda_{i}x_{i} + x_{i+1}$, and moreover we substitute $x_{i+1} = f_{i+1}(x_{r+1}, \ldots, x_{n})$. Since the terms of $f_{i+1}$, of the form $c_{i+1, \mathbf{e}_{r+j} + \mathbf{e}_{r+k}}$, are already uniquely determined, we obtain the same linear equations in the terms $c_{i, \mathbf{e}_{r+j} + \mathbf{e}_{r+k}}$ as in the case in which $x_{i}$ is an eigenvector; the only difference is that we have extra constants, and those again do not change the answer to the question of uniqueness.

The above argument gives us unique quadratic coefficients for each $f_{i}$, from $i = 1$ to $r$. For our third and final induction, suppose we have already uniquely determined all coefficients of every $f_{i}$ up to order $e$, where $e \geq 2$. In $\varphi_{i}$, instead of subtracting $\lambda_{i}x_{i}$ as before, we replace every occurrence of $x_{1}, \ldots, x_{r}$ with $f_{1}, \ldots, f_{r}$, which has the same effect but may be clearer in higher degree. Whenever $x_{1}, \ldots, x_{r}$ occurs in a term of degree more than $1$, we get an expression whose degree-$e$ term depends only on lower-degree $c_{i, \mathbf{\alpha}}$s; concretely, expressions such as $f_{1}^{2}$ and $f_{1}x_{r+1}$ have degree-$e$ terms depending only on degree-less-than-$e$ terms of $f_{1}$. Thus, on the level of degree-$e$ terms, we again obtain a left-hand side equal to either $\lambda_{i}f_{i}$ or $\lambda_{i}f_{i} + f_{i+1}$. As before, we can apply the second induction and reduce to the case of $\lambda_{i}f_{i}$.

We have an equation, true on the level of degree-$e$ terms: $$\lambda_{i}f_{i}(x_{r+1}, \ldots, x_{n}) = f_{i}(\lambda_{r+1}x_{r+1} + (1-\delta_{s1})x_{r+2}, \ldots, \lambda_{n}x_{n})$$ Here, $1-\delta_{s1}$ is $1$ when the Jordan block has length $s > 1$ but $0$ if the Jordan block is trivial. But now we reduce to the case of the first induction: whenever $r + j_{1}, \ldots, r + j_{l}$ are at the heads of their respective Jordan blocks, we have a unique $c_{i, \mathbf{e}_{r + j_{1}} + \ldots + \mathbf{e}_{r + j_{l}}}$, and thence we can uniquely construct the other coefficients, incrementing the index within one Jordan block by one at a time. Here we are using the fact that $\lambda_{i} - \lambda_{r+1}^{\alpha_{r+1}}\ldots \lambda_{n}^{\alpha_{n}} \neq 0$ whenever every $\alpha_{j} \geq 0$ and at least one $\alpha_{j} > 0$, i.e. that $\lambda_{i}$ is not in the multiplicative semigroup generated by $\lambda_{r+1}, \ldots, \lambda_{n}$.

Finally, the equations $x_{i} = f_{i}$ define a formal conjugacy $x_{i} \mapsto x_{i} - f_{i}$. In other words, since the system of formal power series equations $x_{i} = f_{i}$ is invariant under $\varphi$, the system $x_{i} = 0$ is invariant under

\begin{align*}
& (\varphi_{1}(x_{1} + f_{1}, \ldots, x_{r} + f_{r}, \ldots, x_{n}) - f_{1}(\varphi_{r+1}(x_{1} + f_{1}, \ldots, x_{r} + f_{r}, \ldots, x_{n}), \\
& \ldots, \\
& \varphi_{r}(x_{1} + f_{1}, \ldots, x_{r} + f_{r}, \ldots, x_{n}) - f_{r}(\varphi_{r+1}(x_{1} + f_{1}, \ldots, x_{r} + f_{r}, \ldots, x_{n}), \\
& \ldots, \\
& \varphi_{n}(x_{1} + f_{1}, \ldots, x_{r} + f_{r}, \ldots, x_{n}))
\end{align*}
\noindent where bare $f_{i}$ means $f_{i}(x_{r+1}, \ldots, x_{n})$.\end{proof}

We are interested in analytic subvarieties and not just formal ones. For this, we need to ensure that the equations defining $V$, which we call $f_{i}$, have positive radius of convergence.

\begin{lem}\label{anal1}Suppose $F$ is a complete valued field. In the situation of Lemma~\ref{formsep}, suppose that there is a positive real bound $c$ such that we have $|\lambda_{i} - \lambda_{r+1}^{\alpha_{r+1}}\ldots\lambda_{n}^{\alpha_{n}}| > c^{1/(\alpha_{r+1} + \ldots + \alpha_{n})}$ whenever $\alpha_{j} \geq 0$ with at least one $\alpha_{j} > 0$, and $i = 1, \ldots, r$. Then the power series $f_{i}$ defining $V$ all converge in the polydisk $|x_{j}| < c^{s}$ where $s$ is the maximum size of a Jordan block in $x_{r+1}, \ldots, x_{n}$.\end{lem}

\begin{proof}When we constructed $f_{i}$ in the proof of Lemma~\ref{formsep}, we used addition, subtraction, and multiplication in many places, but division only in the step where we show that, in the linear equations defining the coefficients of $f_{i}$, these coefficients (treated as variables) themselves have coefficients equal to $\lambda_{r+1}^{\alpha_{r+1}}\ldots\lambda_{n}^{\alpha_{n}} - \lambda_{i}$. The constant coefficient in these defining equations comes from three sources: other degree-$e$ terms in the same Jordan block; degree-$e$ terms in $\varphi_{i}$; and, in the expression $f_{i}(\varphi_{r+1}, \ldots, \varphi_{n})$, terms coming from plugging in nonlinear coefficients of $\varphi_{j}$, $j > r$, into degree-less-than-$e$ terms of $f_{i}$.

Observe that in a Jordan block of size $s$, and with degree-$e$ coefficients, we need $e(s-1)$ increments by $1$ to move from the first element to the last element of the block. Each increment, we divide by another number, which may be as low as $\sqrt[e]{c}$, so overall we divide by $c^{s-1}$. We also divide by $c$ in the move from degree $e$ to degree $e+1$. Finally, the coefficients of $\varphi_{i}$, except possibly $\lambda_{i}x_{i}$, have absolute value bounded by $1$ by our choice of coordinates. This means all sources of the constant have absolute value bounded as the lemma requires. If $F$ is nonarchimedean, this immediately gives us a bound of $1/c^{es}$ on the absolute value of every degree-$e$ coefficient of $f_{i}$, which gives us the required polyradius of convergence.

If in contrast $F$ is complex, then we need to check that the number of summands comprising the constant of the defining equations does not grow too quickly. In fact, if we can bound the number of summands by any function that grows less than exponentially in $e$, then by standard convergence tests for power series we will get the required polyradius of convergence. Now, $\varphi_{i}$ yields a single term per polydegree; the Jordan block gives us one choice of variable to go back by per degree, so we overall get $e$ summands; and, for any degree $e' < e$, the number of terms of $f_{i}$ of degree $e'$ as well as the number of nonlinear coefficients of the $\varphi_{j}$s that we can plug in grow polynomially, since there are $(e+n-r)!/e!(n-r)!$ coefficients of degree up to $e$, and this grows as $O(e^{n-r})$.\end{proof}

\begin{rem}\label{smooth}If the power series $f_{i}$ all converge, then the subvariety $V$ is nonsingular, since $$\frac{\partial (x_{i} - f_{i})}{\partial x_{j}} = \delta_{ij}.$$\end{rem}

The conditions of Lemma~\ref{anal1} are satisfied in many natural cases, particularly in the nonarchimedean case. Suppose $F$ is nonarchimedean. If we choose $x_{1}, \ldots, x_{r}$ such that $\lambda_{1}, \ldots, \lambda_{r}$ are all the repelling multipliers, then the bound $c = \min_{i = 1}^{r}\{|\lambda_{i}|\} > 1$ works. If $x$ has no repelling directions, and we choose $x_{1}, \ldots, x_{r}$ to be either all of the indiffierent directions or all of the attracting directions, then the bound $c = 1$ works. Observe that once $c \geq 1$, there is no need at all to take the $e$th root. Observe also that if there are no nontrivial Jordan blocks, there is again no need to take the $e$th root.

If $F = \mathbb{C}$, $x$ has no repelling directions, and we choose $x_{1}, \ldots, x_{r}$ to be all of the indifferent directions, then the conditions of the lemma are satisfied, even after taking $e$th roots. In the worst case, when the multipliers lie on the same ray from the origin, we need to find $c$, depending the positive real $\lambda$, such that $|1 - \lambda^{e}| > \sqrt[e]{c}$. This can be rewritten as $(1 - \lambda^{e})^{e} > c$. Now $(1 - \lambda^{e})^{e} > 1 - e\lambda^{e}$, and the quantity $e\lambda^{e}$ is bounded when $\lambda < 1$, so we can indeed find a bound $c$.

If $F = \mathbb{C}$ and we have a mix of rationally and irrationally indifferent multipliers, the situation is more delicate, mirroring the question of linearization in one variable near irrationally indifferent points.

\begin{rem}It is well-known to complex dynamicists that if $F = \mathbb{C}$, $\lambda_{1}, \ldots, \lambda_{r}$ are attracting, and $\lambda_{r+1}, \ldots, \lambda_{n}$ are repelling, then it is possible to construct an two analytic $V$s as in Lemmas~\ref{formsep} and~\ref{anal1}, one tangent to $\lambda_{1} = \ldots = \lambda_{r} = 0$ and one tangent to $\lambda_{r+1} = \ldots = \lambda_{n} = 0$. See~\cite{Yoc2} for a survey of this result. However, there does not seem to be any result in the literature separating indifferent directions from attracting or repelling ones, nor is there any result over nonarchimedean fields.\end{rem}

Finally, in analogy with the terminology of attracting, repelling, and indifferent fixed points, we can study the behavior of points near the fixed subvariety $V$ constructed in Lemma~\ref{formsep}. Here, for simplicity, we use the fact that $V$ is locally conjugate to a linear subvariety. However, since every variety is locally linearizable at every smooth point, this is really a result about general fixed subvarieties at smooth fixed points.

\begin{prop}\label{V}Let $\varphi = (\varphi_{1}, \ldots, \varphi_{n})$ be analytic on the open unit polydisk. Suppose that $x = (0, \ldots, 0)$ is fixed under $\varphi$, that the matrix of linear terms of $\varphi$ is already in Jordan canonical form with multipliers $\lambda_{1}, \ldots, \lambda_{n}$, and that $V$, defined by $x_{1} = \ldots = x_{r} = 0$, is also fixed under $\varphi$. Then $\varphi|_{V}$, the restriction of $\varphi$ to $V$, also fixes $x$, with multipliers $\lambda_{r+1}, \ldots, \lambda_{n}$.\end{prop}

\begin{proof}This is an immediate consequence of the fact that $$\varphi|_{V} = (\varphi_{r+1}(0, \ldots, 0, x_{r+1}, \ldots, x_{n}), \ldots, \varphi_{n}(0, \ldots, 0, x_{r+1}, \ldots, x_{n}))$$\end{proof}

The multipliers $\lambda_{r+1}, \ldots, \lambda_{n}$ govern the behavior of points on $V$, near $x$. The converse to this should be that the multipliers $\lambda_{1}, \ldots, \lambda_{r}$ govern the behavior of points near $V$: if the multipliers $\lambda_{1}, \ldots, \lambda_{r}$ are attracting then $V$ should be attracting. Our next proposition states that this is true, if an additional condition is met.

\begin{prop}\label{rate}Let $F$ be a complete nonarchimedean valued field, and let $\varphi = (\varphi_{1}, \ldots, \varphi_{n})$ be defined over $F$ and analytic on the open unit polydisk. Suppose that $x = (0, \ldots, 0)$ is fixed under $\varphi$, that the matrix of linear terms of $\varphi$ is already in Jordan canonical form with multipliers $\lambda_{1}, \ldots, \lambda_{n}$, and that $V$, defined by $x_{1} = \ldots = x_{r} = 0$, is also fixed under $\varphi$. Suppose also that $|\lambda_{r+1}|, \ldots, |\lambda_{n}| \leq 1$. If $|\lambda_{1}|, \ldots, |\lambda_{r}| < 1$ then $V$ is attracting, in the sense that there exists an open neighborhood $D \ni x$ in which for every $z \in D$, the minimal distance from $\varphi^{k}(z)$ to $V$ tends to $0$ as $k \to \infty$. Moreover, for large enough $k$, the minimal distance from $\varphi^{k}(z)$ to $V$ is proportional to $\max_{i = 1}^{r}\{|\lambda_{i}|^{k}\}$ provided at least one of $\lambda_{1}, \ldots, \lambda_{r}$ is nonzero.\end{prop}

\begin{proof}First, we use the fact that the multipliers are all non-repelling and apply linear coordinate-change to ensure that all the coefficients of $\varphi$ have absolute value at most $1$. Write $z = (z_{1}, \ldots, z_{n})$, with $|z_{i}| \leq c < 1$, and write $\varphi(z) = (\varphi_{1}(z), \ldots, \varphi_{n}(z))$. For any $i = 1, \ldots, n$, $\varphi_{i}(z)$ is the sum of terms of absolute value bounded by $c$. This means $|\varphi_{i}(z)| \leq c$. Now, let us assume $i = 1, \ldots, r$, and work Jordan block by Jordan block. Suppose $x_{1}, \ldots, x_{s}$ form a Jordan block. By assumption, $\lambda_{1} = \ldots = \lambda_{s}$ and this common multiplier has absolute value less than $1$. Now, $\varphi_{s} = \lambda_{1}z_{s} + \mbox{(higher-order terms)}$, so $|\varphi_{s}| \leq \max\{|\lambda_{1}|c, c^{2}\} < c$. For simplicity, let us choose $c$ to be small enough that $|\lambda_{1}|c > c^{2}$. In contrast, if $i = 1, \ldots, s-1$, then $\varphi_{i} = \lambda_{1}z_{i} + z_{i+1} + \mbox{(higher-order terms)}$ and then if $|z_{1}| = \ldots = |z_{n}| = c$ then $|\varphi_{i}| = c$ as well.

Now, write $\varphi^{k}(z) = (\varphi_{1, k}(z), \ldots, \varphi_{n, k}(z))$. We have $|\varphi_{i}| \leq c$ for all $i$ but also $|\varphi_{s}| \leq |\lambda_{1}|c$. Now $|\varphi_{s, 2}| \leq \max\{|\lambda_{1}|^{2}c, c^{2}\}$ and $|\varphi_{s-1, 2}| \leq |\lambda_{1}|c$. We now need to choose $c$ to be even smaller; this is because $\varphi_{s}$ may contain such terms as $x_{1}^{2}$ and $x_{1}x_{r+1}$, which still have absolute value bounded only by $c^{2}$ and nothing smaller. We may repeat this process $s$ times, obtaining $|\varphi_{i, s}| \leq \max\{|\lambda_{1}|^{i}c, c^{2}\}$ for $i = 1, \ldots, s$. This requires choosing $c$ such that $|\lambda_{1}|^{s}c > c^{2}$. Roughly at this point, the process ends: once we choose $c$ to be smaller than any $|\lambda_{i}|^{s_{i}}$, where $s_{i}$ is the size of the Jordan block containing $x_{i}$, any higher-order term in $\varphi_{s}$ will have absolute value bounded by a smaller constant than $c^{2}$, namely $\max_{i = 1}^{r}\{|\lambda_{i}|\}c^{2}$, obtained at such terms as $x_{1}x_{r+1}$, i.e. quadratic terms with one term at the head of its Jordan block and the other in the range $x_{r+1}, \ldots, x_{n}$ (since $V$ is fixed, terms such as $x_{r+1}^{2}$ cannot occur in $\varphi_{1}, \ldots, \varphi_{r}$).

If $k > s$ we then obtain $|\varphi_{i, k}| \leq |\lambda_{1}|^{k-s+i}c$, again for $i = 1, \ldots, s$. This tends to $0$, at a rate of at worst a factor of $|\lambda_{1}|$ per iteration. Similarly, for any $i = 1, \ldots, r$, we have $|\varphi_{i, t}| \leq |\lambda_{i}|^{t}c_{i}$ where $c_{i}$ depends on $c$ and on the position of $x_{i}$ in its Jordan block. This again tends to $0$, at (at worst) the required rate. Finally, let us show that this is the exact rate at which $\varphi^{k}(z)$ approaches $V$. Assume without loss of generality that the $\max_{i = 1}^{r}\{|\lambda_{i}|\} = \lambda_{1}$. We will show that $|\varphi_{s, k}| = |\lambda_{1}|^{k}c$, provided $c < |\lambda_{1}|^{\max\{s_{i}\}-1}$. As before, $\varphi_{s} = \lambda_{1}z_{s} + \mbox{(higher-order terms)}$. After $k \geq \max\{s_{i}\}$ iterations, the higher-order terms are bounded by $|\lambda_{1}|^{k-\max\{s_{i}\}+1}c^{2}$. But now $|\lambda_{1}|^{k-\max\{s_{i}\}+1}c^{2} < |\lambda_{1}|^{k}c$. Thus, on the orbit of $z$, the maximal absolute value among the monomials of $\varphi_{s}$ always occurs at the linear term, so $|\varphi_{s, k}| = |\lambda_{1}|^{k}c$, as required.\end{proof}

\begin{rem}Proposition~\ref{rate} is also true in the complex case, but the condition that $|\lambda_{r+1}|, \ldots, |\lambda_{n}| \leq 1$ needs to be tightened: we require $x$ to be in the Fatou set of $\varphi|_{V}$, or else we require $x$ to be rationally indifferent and $z$ to be such that it approaches the Fatou set of $\varphi|_{V}$. Rather than state a full result, we will deal with this case in Section~\ref{zhang}, in which case we will have $n = 2$ and $r = 1$, reducing the question to much simpler one-dimensional complex dynamics.\end{rem}

\begin{rem}If $\lambda_{1} = \ldots = \lambda_{r} = 0$, then according to Proposition~\ref{rate}, if $z$ is near $x$ then it approaches $V$. Indeed, if $|z_{i}| \leq c$, then $|\varphi_{i, k}(z)| \leq c^{k}$ as long as $i = 1, \ldots, r$. Observe that we do not obtain any superattraction here: in $1$ dimension, if $x$ is a superattracting fixed point, mapping to itself with degree $e > 1$, then any $z$ near $x$ satisfies $|\varphi^{k}(z) - x| \leq C|z - x|^{e}$ where $C$ is a constant. But here, $\varphi^{k}(z)$ approaches $V$ geometrically, by a factor of $\max\{z_{r+1}, \ldots, z_{n}\}$, provided $\lambda_{r+1}, \ldots, \lambda_{n}$ are all indifferent. It is still true that we can get arbitrarily fast attraction by choosing $z$ to have small coefficients $z_{r+1}, \ldots, z_{n}$. We say in this case that $x$ is superattracting, but $V$ is not. This is one way to see that, even if $\lambda_{1} = \ldots = \lambda_{r} = 0$, $V$ is not necessarily a critical subvariety, mapping to itself with degree more than $1$. For example, consider $\varphi(x_{1}, x_{2}) = (x_{1}^{2} + x_{1}x_{2}, x_{2} + x_{2}^{2})$: the subvariety $x_{1} = 0$ is fixed, and at the fixed point $x = (0, 0)$, $\lambda_{1} = 0$, but the subvariety $x_{1} = 0$ is not critical.\end{rem}

\section{Isolated Periodic Points}\label{isolated}

In this section, we prove Theorem~\ref{isolatedper}. Throughout this section, $F$ is a complete, algebraically closed, nonarchimedean valued field; our main model is $\mathbb{C}_{p}$. We denote by $p$ the residue characteristic of $F$. Initially we will not make any assumption on the characteristic of $F$, but eventually we will require $F$ to have characteristic $0$. We denote the maximal ideal of $F$ by $\mathfrak{m}$, and pass back and forth between the absolute value on $F$, $|\cdot|$, and the valuation on $F$, $v(\cdot) = -\log_{p}|\cdot|$.

We apply the results of Section~\ref{separating} by reducing the problem to proving that periodic points are isolated in the indifferent case. We will prove the following theorem:

\begin{thm}\label{isolatedindiff}Let $\varphi$ be an analytic map, defined over $F$ of characteristic $0$, from the open unit polydisk in $n$ dimensions to itself. If $x = (0, \ldots, 0)$ is a fixed point all of whose multipliers are indifferent, then either $x$ is isolated as a periodic point of $\varphi$, or, for some iterate $\varphi^{k}$, there exists a pointwise fixed subvariety of positive dimension passing through $x$.\end{thm}

\begin{proof}[Proof of Theorem~\ref{isolatedper} assuming Theorem~\ref{isolatedindiff}.] Dehomogenize and choose coordinates such that $x = (0, \ldots, 0)$, all coefficients of $\varphi$ have absolute value at most $1$, the linear terms of $\varphi$ are in Jordan canonical form, with $|\lambda_{1}|, \ldots, |\lambda_{r}| < 1$ and $|\lambda_{r+1}|, \ldots, |\lambda_{n}| = 1$. The conditions of both Lemmas~\ref{formsep} and~\ref{anal1} hold, since for all $i = 1, \ldots, r$ and for all integers $\alpha_{r+1}, \ldots, \alpha_{n}$, we have $$|\lambda_{i} - \lambda_{r+1}^{\alpha_{r+1}}\ldots\lambda_{n}^{\alpha_{n}}| = 1.$$ Thus, there exists a fixed analytic subvariety $V$ defined by the equations $x_{i} = f_{i}$ where $f_{1}, \ldots, f_{r} \in F[[x_{r+1}, \ldots, x_{n}]]$, and we can do a local analytic coordinate change such that $V$ is in fact linear, of the form $x_{1} = \ldots = x_{r} = 0$. This coordinate change converges in the open unit polydisk, and the new power series still have coefficients with absolute value at most $1$, since, as noted in the proof of Lemma~\ref{anal1}, we only ever divide by quantities of the form $\lambda_{i} - \lambda_{r+1}^{\alpha_{r+1}}\ldots\lambda_{n}^{\alpha_{n}}$, which have absolute value $1$ in this case.

Now, we check that there are no periodic points for $\varphi$ near $x$, except possibly on $V$. For this, we apply Proposition~\ref{rate}, whose conditions are satisfied since none of the multipliers is repelling. In a neighborhood of $x$, which we may in fact take to be the open unit polydisk, points are attracted to $V$: any point near $x$ not on $V$, of the form $(z_{1}, \ldots, z_{n})$, with $(z_{1}, \ldots, z_{r}) \neq (0, \ldots, 0)$, will approach $V$. Clearly, such points cannot be periodic.

It remains to be shown that there is a neighborhood of $x$ whose intersection with $V$ contains no periodic points other than $x$ itself. We have a map $\varphi$ from the open unit polydisk in $r$ dimensions to itself. Moreover, $x$ is a fixed point, whose multipliers are all indifferent, by Proposition~\ref{V}. We apply Theorem~\ref{isolatedindiff}, and obtain that either there exists a positive-dimension pointwise fixed subvariety under some $\varphi^{k}$, or $x$ is isolated as a periodic point as required. However, as a local deformation of an algebraic map, $\varphi$ inherits the following property: algebraic morphisms from $\mathbb{P}^{n}$ to itself never have positive-dimension pointwise fixed subvarieties, and instead have finitely many fixed points (and finitely many periodic points of a fixed period). This is well-known to dynamicists; see for example~\cite{PST}. Since there is no pointwise fixed subvariety of positive dimension, the other conclusion of Theorem~\ref{isolatedindiff} is obtained, and $x$ is indeed isolated.\end{proof}

Before going on to prove Theorem~\ref{isolatedindiff}, we will reprove a one-dimensional version of the theorem. This is not new; it goes back at least as far as Rivera-Letelier (\cite{RL}; see also~\cite{LRL}). However, we will explicitly use a Newton polygon technique that we will generalize in order to prove the same result in several variables.

Recall that the Newton polygon of a power series $g(z) = \sum_{i=0}^{\infty}a_{i}z^{i} \in F[[z]]$ is the lower convex hull in $\mathbb{R}^{2}$ of the set $\{(i, v(a_{i}))\}$. It is a finite or countable union of line segments. We define the length of a segment to be the length of its projection down to the horizontal axis; in case $a_{0} = 0$, and the minimal index such that $a_{i} \neq 0$ is $e > 0$, we have a vertical segment, which we by convention say has length $e$ and slope $-\infty$. The theorem of the Newton polygon states that $g(z)$ has a root of valuation $m$ only if the Newton polygon of $g$ has a segment of slope $-m$, and moreover, if $g(z)$ has a segment of slope $-m$ and finite length $k$ then $g$ has exactly $k$ roots of valuation $m$.

We now have,

\begin{thm}\label{rl}(Rivera-Letelier) Let $\varphi$ be a power series over $F$ in one variable, with coefficients of absolute value at most $1$. Suppose $0$ is an indifferent fixed point of $\varphi$, with multiplier $\lambda$. If $F$ has characteristic $0$, or if $\lambda$ is not a root of unity, then either $0$ is isolated as a periodic point, or $\varphi$ has an iterate that is equal to the identity map.\end{thm}

\begin{proof}First, replacing $\varphi$ with an iterate if necessary, we assume that $\lambda \equiv 1 \mod \mathfrak{m}$; if $F$ is so big that its residue field is larger than $\overline{\mathbb{F}_{p}}$ and $\lambda$ reduces to a transcendental element mod $\mathfrak{m}$, then as we will shortly see, the theorem is trivial.

Now, we will consider the Newton polygons of $\varphi^{k}(z) - z$, and see that their negative finite slopes cannot be too steep. Throughout this proof, we assume that $\varphi^{k}(z) - z$ is never the zero series, which would imply $\varphi^{k}(z)$ is the identity.

We look at the first nonzero term of $\varphi^{k}(z) - z$. Suppose first that $\lambda$ is not a root of unity, so that we may write it as $\lambda = 1 + b_{1}\pi^{e_{1}} + b_{2}\pi^{e_{2}} + \ldots$ where $v(\pi) = 1$ and $e_{i} \in \mathbb{R}$ (if $F = \mathbb{C}_{p}$ then in fact $e_{i} \in \mathbb{Q}$). The first nonzero term of $\varphi^{k}(z) - z$ is $\lambda^{k} - 1$. Note that if there is a root $z$ of valuation $m$, then there must be $k$ such roots, corresponding to the entire cycle of $z$ (if $z$ is actually of lower period, say $l$, with $l | k$, then we will see it as a root of $\varphi^{l}(z) - z$). Therefore, if we show that $v(\lambda^{k} - 1) \leq Ck$ where $C$ is some constant, then it will bound the valuation of any nonzero periodic point of $\varphi$ by $C$, as required. Note that if $\lambda \not\equiv 1 \mod \mathfrak{m}$ then $v(\lambda^{k} - 1) = 0$ and therefore we have a trivial bound.

Suppose now that $p \nmid k$. It is not hard to see that $\lambda^{k} = 1 + kb_{1}\pi^{e_{1}} + \ldots$, so that $v(\lambda^{k} - 1) = e_{1}$. The case of interest is then what happens when $k$ is a power of $p$. We will assume $k = p$ to give a recursive formula. If $F$ has characteristic $p$, then $\lambda^{p} = 1 + b_{1}^{p}\pi^{pe_{1}} + \ldots$ by Frobenius, and $\lambda^{p^{s}} = 1 + b_{1}^{p^{s}}\pi^{p^{s}e_{1}} + \ldots$ Therefore, $v(\lambda^{p^{s}} - 1) = p^{s}e_{1}$, and we obtain $v(\lambda^{k} - 1) \leq ke_{1}$, as required.

From now on, we assume $F$ has characteristic $0$, in which case we might as well choose $\pi = p$. We have $$\lambda^{p} = (1 + b_{1}p^{e_{1}} + \ldots)^{p} = 1 + pb_{1}p^{e_{1}} + \ldots + b_{1}^{p}p^{pe_{1}} + \ldots$$ Which term is dominant, $b_{1}p^{e_{1} + 1}$ or $b_{1}^{p}p^{pe_{1}}$, depends on the valuation of $e_{1}$: if $v(e_{1}) > 1/(p-1)$ then $b_{1}p^{e_{1} + 1}$ dominates, and if $v(e_{1}) < 1/(p-1)$ then $b_{1}^{p}p^{pe_{1}}$ dominates. If $v(e_{1}) = 1/(p-1)$ then we have a resonance, and $v(\lambda^{p} - 1)$ can be arbitrarily large, since we can write the $p$th roots of unity in $\mathbb{C}_{p}$ in this form with $e_{1} = 1/(p-1)$. But we are assuming that $\lambda$ is not a root of unity, so $v(\lambda^{p} - 1)$ is finite. This implies that we can replace $\varphi$ with a $p$-power iterate until $e_{1} > 1/(p-1)$. In this case, we have $\lambda^{p} = 1 + b_{1}p^{e_{1} + 1} + \ldots$ and $\lambda^{p^{s}} = 1 + b_{1}p^{e_{1} + s} + \ldots$ Thus $v(\lambda^{k} - 1) \leq e_{1} + \log_{p}k$.

Now, we assume $\lambda$ is a root of unity, and $F$ has characteristic $0$. We replace $\varphi$ by an iterate to fix $\lambda = 1$. This means that $\varphi(z) = z + a_{e}z^{e} + \ldots$ and the lowest nonzero term is $a_{e}$. Let us compute a formula for $\varphi^{k}(z)$. We prove by induction that in fact $\varphi^{k}(z) = z + ka_{e}z^{e} + \ldots$ This is clearly true if $k = 1$, so let us assume it is true for $k$, and prove it for $k+1$. We have $$\varphi^{k+1}(z) = \varphi^{k}(\varphi(z)) = \varphi(z) + ka_{e}(\varphi(z))^{e} + \ldots = (z + a_{e}z^{e} + \ldots) + ka_{e}(z + \ldots)^{e} + \ldots = z + (k+1)a_{e}z^{e} + \ldots$$ Thus, the first nonzero term of $\varphi^{k}(z) - z$ is $ka_{e}z^{e}$, which has valuation $v(k) + v(a_{e})$. As in the irrationally indifferent case, we do not need to bound this valuation, but only show that it grows at worst as $O(k)$. But clearly $v(k) + v(a_{e}) \in O(k)$ -- in fact, the growth rate is logarithmic in $k$, as in the characteristic-$0$ irrationally indifferent case.\end{proof}

We will extend the arguments in the proof of Theorem~\ref{rl} to higher dimension using the Newton polytope as developed by Rabinoff in~\cite{Rab}. The Newton polytope is the natural higher-dimensional generalization of the Newton polygon: the Newton polytope of a power series $\sum c_{\mathbf{\alpha}}x_{1}^{\alpha_{1}}\ldots x_{n}^{\alpha_{n}}$ is an $n$-dimensional linear complex in $\mathbb{R}^{n+1}$ obtained as the lower convex hull of the set $\{(\alpha_{1}, \ldots, \alpha_{n}, v(c_{\mathbf{\alpha}}))\}$. Here we label the axes of $\mathbb{R}^{n+1}$ by $w_{1}, \ldots, w_{n}, y$, and say ``lower'' to mean ``lower along the $y$-axis.''

The segments of the Newton polytope have slopes, which we define as $\partial y/\partial w_{i}$. The theorem of the Newton polygon is that, in its domain of convergence, a single-variable power series $g(z)$ has roots of valuation $m$ ($m \in \mathbb{R} \cup \{\infty\}$) only if the Newton polygon of $g$ has a segment of slope $-m$, and conversely if the Newton polygon of $g$ has a segment of slope $-m$ and horizontal length $k$ then $g$ has exactly $k$ roots of valuation $m$. A multivariable power series $g(x_{1}, \ldots, x_{n})$ has zeros of polyvaluation $(m_{1}, \ldots, m_{n})$ only if, at this polyvaluation, $g$ has more than $1$ term of maximum absolute value (and if it has finitely many such terms, this is if and only if); in Newton polytope language, this means $g$ has a segment of polyslope $(\partial w_{1}: \ldots: \partial w_{n}: \partial y)$ whose dot product with the projective vector $(m_{1}: \ldots: m_{n}: 1)$ is zero.

Rabinoff proves a generalization of the theorem of the Newton polygon to intersections of $n$ power series in $n$ variables, provided those intersections are proper; see Definition 11.2, Example 11.3, and Theorem 11.7 of~\cite{Rab}. We will now go over the exact statement of the theorem of the Newton polytope.

\begin{defn}\label{MV}Let $P_{1}, \ldots, P_{n}$ be polytopes in $\mathbb{R}^{n}$; for our purposes, a polytope is a closed, bounded, convex region, defined by finitely many linear equations and inequalities, and may have any dimension from $0$ (in which case it is a point) to $n$. We set $$c_{1}P_{1} + \ldots + c_{n}P_{n} = \{c_{1}p_{1} + \ldots + c_{n}p_{n}: p_{i} \in P_{i}\}$$ We have a volume function $$V_{P_{1}, \ldots, P_{n}}: \mathbb{R}^{n} \to \mathbb{R}, V_{P_{1}, \ldots, P_{n}}(c_{1}, \ldots, c_{n}) = \Vol (c_{1}P_{1} + \ldots + c_{n}P_{n})$$ It is a homogeneous polynomial of degree $n$. We define the \textbf{mixed volume} of $P_{1}, \ldots, P_{n}$ to be the $c_{1}\ldots c_{n}$-coefficient of the polynomial $V_{P_{1}, \ldots, P_{n}}$; we denote the mixed volume by $MV(P_{1}, \ldots, P_{n})$.\end{defn}

\begin{rem}The mixed volume is clearly translation-invariant. Observe now that if each $P_{i}$ is one-dimensional, and is a line segment from the origin to the point $v_{i} \in \mathbb{R}^{n}$, then $MV(P_{1}, \ldots, P_{n}) = \det(v_{1}, \ldots, v_{n})$. Observe also that if any $P_{i}$ is a point, then the mixed volume is zero, since $V_{P_{1}, \ldots, P_{n}}$ then does not depend on $c_{i}$, which means all terms including $c_{i}$, including $c_{1}\ldots c_{n}$, are zero. Finally, observe that if $P_{1} = \ldots = P_{n}$ then $c_{1}P_{1} + \ldots + c_{n}P_{n} = (c_{1} + \ldots + c_{n})P_{1}$ and then, reading off the mixed coefficient of $(c_{1} + \ldots + c_{n})^{n}$, we obtain $MV(P_{1}, \ldots, P_{n}) = n!\Vol P_{1}$.\end{rem}

\begin{thm}\label{rabinoff}(Rabinoff) Let $g_{1}, \ldots, g_{n}$ be power series over $F$ in $n$ variables. For each polyvaluation $(m_{1}, \ldots, m_{n})$, we define $P_{i}$ to be the convex hull in $\mathbb{R}^{n}$ of the points $(\alpha_{1}, \ldots, \alpha_{n})$ over all vectors $\mathbf{\alpha}$ such that $v(c_{i, \mathbf{\alpha}}) + \alpha_{1}m_{1} + \ldots + \alpha_{n}m_{n}$ is minimized; alternatively, this is the projection to $\mathbb{R}^{n}$ of the relevant segment of the Newton polytope of $g_{i}$. Suppose that $g_{1}, \ldots, g_{n}$ intersect properly. Then, counting multiplicity, the number of common zeros of $g_{1}, \ldots, g_{n}$ of polyvaluation $(m_{1}, \ldots, m_{n})$ equals $MV(P_{1}, \ldots, P_{n})$.\end{thm}

When we generalize the proof of Theorem~\ref{rl} to prove Theorem~\ref{isolatedindiff}, we will need to pass between the volumes of the polytopes $P_{i}$ and their mixed volume. It is the mixed volume that grows as $k$, and the volumes of the polytopes that, as we will see, grow logarithmically with $k$. Therefore, we need to directly relate $\Vol P_{1}, \ldots, \Vol P_{n}$ and $MV(P_{1}, \ldots, P_{n})$.

\begin{prop}\label{volrel}We have the inclusion-exclusion formula $$MV(P_{1}, \ldots, P_{n}) = \Vol(P_{1} + \ldots + P_{n}) - \sum_{i = 1}^{n}\Vol(P_{1} + \ldots + \hat{P_{i}} + \ldots + P_{n}) + \ldots + (-1)^{n}\sum_{i = 1}^{n}\Vol(P_{i}).$$\end{prop}

\begin{proof}The central observation is that for each $j = 1, \ldots, n$, the quantity $\Vol(P_{1} + \ldots + P_{j})$ is equal to the sum of all coefficients of $V_{P_{1}, \ldots, P_{n}}(c_{1}, \ldots, c_{n})$ that contain powers of only $c_{1}, \ldots, c_{j}$. This means that the right-hand side of the formula in the statement of the proposition is a sum of coefficients of $V_{P_{1}, \ldots, P_{n}}(c_{1}, \ldots, c_{n})$.

We now consider each term that contains the variables $c_{1}, \ldots, c_{j}$ and only them, and figure out which of the volume formulas in the right-hand side of the statement of the proposition it is contained in. Clearly, the volume formula must include $P_{1}, \ldots, P_{j}$; it may also include any $P_{i}$ with $i > j$. Thus, we have $2^{n-j}$ different formulas that include these terms, corresponding to elements of the power set $\mathcal{P}(\{P_{j+1}, \ldots, P_{n}\})$. Now, the inclusion-exclusion formula in the proposition takes coefficient $1$ when the number of omitted variables is even and $-1$ when the number of omitted variables is odd. Whenever $j < n$, exactly half the elements of $\mathcal{P}(\{P_{j+1}, \ldots, P_{n}\})$ correspond to an even number of variables and exactly half correspond to an odd number of variables. Thus, each coefficient of a term that contains the variables $c_{1}, \ldots, c_{j}$ and only them cancels out on the right-hand side.

We are left with just one coefficient, namely the mixed term $c_{1}\ldots c_{j}$. Here, the only summand on the right-hand side that contains this coefficient is $\Vol(P_{1} + \ldots + P_{n})$. Thus, it occurs exactly once on the right-hand side. So we read off this coefficient, which is by definition $MV(P_{1}, \ldots, P_{n})$.\end{proof}

\begin{cor}\label{volbound}Suppose the polytope $P$ contains $P_{1}, \ldots, P_{n}$. Then $MV(P_{1}, \ldots, P_{n}) \leq n^{n+1}\Vol P$.\end{cor}

\begin{proof}By Proposition~\ref{volrel}, we have,

\begin{align*}
MV(P_{1}, \ldots, P_{n}) &= \Vol(P_{1} + \ldots + P_{n}) - \sum_{i = 1}^{n}\Vol(P_{1} + \ldots + \hat{P_{i}} + \ldots + P_{n}) + \ldots + (-1)^{n}\sum_{i = 1}^{n}\Vol(P_{i})\\
&\leq \Vol(P_{1} + \ldots + P_{n}) + \sum_{1 \leq i < j \leq n}\Vol(P_{1} + \ldots + \hat{P_{i}} + \ldots + \hat{P_{j}} + \ldots + P_{n}) + \ldots\\
&+ \left(\sum_{i = 1}^{n}\Vol(P_{i}) \mbox{ or } \sum_{1 \leq i < j \leq n}\Vol(P_{i} + P_{j})\right)\\
&\leq \Vol(nP) + {n \choose 2}\Vol((n-2)P) + \ldots + {n \choose k}\Vol((n-k)P) + \ldots\\
&+ \left(n\Vol(P) \mbox{ or } {n \choose 2}\Vol(2P)\right)\\
&= \Vol(P)\left(n^{n} + {n \choose 2}(n-2)^{n} + \ldots + \left(n \mbox{ or } {n \choose 2}2^{n}\right)\right)\\
&\leq \Vol(P)\cdot n^{n}\cdot\left\lfloor\frac{n}{2}+1\right\rfloor\\
&\leq n^{n+1}\Vol(P)
\end{align*}

The $A \mbox{ or } B$ language consistently means $A$ if $n$ is even and $B$ if $n$ is odd.\end{proof}

In the proof of Theorem~\ref{rl}, we use the fact that the valuation of the lowest nonzero term of $\varphi^{k}(z) - z$ grows logarithmically in $k$ if $F$ has characteristic $0$, or at worst linearly if $F$ has characteristic $p$. Let us now show that, if we can show similar growth rates in $n$ dimensions, we can prove Theorem~\ref{isolatedindiff}.

\begin{defn}\label{minpoly}Set the partial order $\preceq$ on $\mathbb{Z}^{n}$ to be $\mathbf{\beta} \preceq \mathbf{\alpha}$ if $\beta_{i} \leq \alpha_{i}$ for $i = 1, \ldots, n$; we say $\mathbf{\beta} \prec \mathbf{\alpha}$ if $\mathbf{\beta} \preceq \mathbf{\alpha}$ and $\mathbf{\beta} \neq \mathbf{\alpha}$. Let $g(x_{1}, \ldots, x_{n}) \in F[[x_{1}, \ldots, x_{n}]]$. We say $x_{1}^{\alpha_{1}}\ldots x_{n}^{\alpha_{n}}$ is a \textbf{lowest-polydegree term} of $g$ if its coefficient is nonzero, and the coefficient of $x_{1}^{\beta_{1}}\ldots x_{n}^{\beta_{n}}$ is zero whenever $\mathbf{\beta} \prec \mathbf{\alpha}$. We will abuse terminology and also call $\mathbf{\alpha}$ a minimal polydegree of $g$. If $g_{1}, \ldots, g_{r} \in F[[x_{1}, \ldots, x_{n}]]$, then we say $x_{1}^{\alpha_{1}}\ldots x_{n}^{\alpha_{n}}$ is a lowest-polydegree term for the family $g_{1}, \ldots, g_{r}$ if its coefficient in one of the $g_{i}$s is nonzero, and the coefficient of $x_{1}^{\beta_{1}}\ldots x_{n}^{\beta_{n}}$ is zero in $g_{1}, \ldots, g_{r}$ whenever $\mathbf{\beta} \prec \mathbf{\alpha}$. If $x_{1}^{\alpha_{1}}\ldots x_{n}^{\alpha_{n}}$ is a lowest-polydegree term for the family $g_{1}, \ldots, g_{r}$, then we say its \textbf{minimal valuation} is its minimum valuation across all $g_{i}$s.\end{defn}

\begin{rem}For every finite collection $g_{1}, \ldots, g_{r}$, there are finitely many terms of lowest polydegree. If any $g_{i}$ has a nonzero constant term, then the constant term is the unique lowest-polydegree term.\end{rem}

\begin{lem}\label{implies}Let $\varphi = (\varphi_{1}, \ldots, \varphi_{n}) \in (F[[x_{1}, \ldots, x_{n}]])^{n}$ be an analytic map, such that $\varphi_{i}(0) = 0$ for all $i$ with all multipliers indifferent, and all the coefficients of all $\varphi$ have nonnegative valuation. Suppose that for every iterate $\varphi^{k} = (\varphi_{1, k}, \ldots, \varphi_{n, k})$, the lowest-polydegree terms of the collection $\varphi_{1, k} - x_{1}, \ldots, \varphi_{n, k} - x_{n}$ have total degree bounded by an absolute constant $e$ and minimal valuation bounded by a function $h(k) \in O(\sqrt[n]{k})$. Then $0$ is isolated as a periodic point, or else $\varphi$ has a positive-dimension pointwise fixed subvariety passing through $0$.\end{lem}

\begin{proof}We are going to intersect $\varphi_{i, k} - x_{i}$, using the volume estimate in Corollary~\ref{volbound}. Since we are ultimately invoking Theorem~\ref{rabinoff}, we assume throughout this proof that the analytic hypersurfaces $\varphi_{i, k} - x_{i}$ intersect properly near $0$.

We need to find some constant $m$, such that if $m_{1}, \ldots, m_{n} > m$, then the system $\varphi_{i, k} - x_{i} = 0$ has no roots of polyvaluation $(m_{1}, \ldots, m_{n})$. If we find roots of small absolute value (in this context, this is the maximal absolute value of all the coordinates), then we shrink $m$; we will show that if we shrink $m$ enough, then we will not find any more roots.

The key fact here is that if $A \in \GL_{n}(\mathcal{O}_{F})$ and $(z_{1}, \ldots, z_{n})$ is a root of the system $g_{1} = \ldots = g_{n} = 0$, then $A^{-1}\mathbf{z}$ is a root of the system $A\mathbf{g} = 0$, and moreover, if $v(z_{i}) \geq m$ for each $i$, then $v(A_{i}\mathbf{z}) \geq m$ for each $i$ as well. Moreover, we can choose $A$ such that each power series $A_{i}\mathbf{g}$ has the same minimal-polydegree terms as the collection $g_{1}, \ldots, g_{n}$, with each $\mathbf{x}^{\mathbf{\alpha}}$ having coefficient of valuation equal to the minimal valuation of $\mathbf{x}^{\mathbf{\alpha}}$. This is because we have $$v\left(\sum_{j = 1}^{n}A_{ij}c_{j, \mathbf{\alpha}}\right) = \min\{v(A_{ij}c_{j, \mathbf{\alpha}})\}$$ as long as we avoid finitely many conditions of the form $$\sum_{j = 1}^{n}A'_{ij}c'_{j, \mathbf{\alpha}} = 0,$$ where $A'_{ij}$ is the reduction mod $\mathfrak{m}$ of $A_{ij}$, $c'_{j, \mathbf{\alpha}}$ is the reduction mod $\mathfrak{m}$ of $c_{j, \mathbf{\alpha}}/p^{s}$, and $s$ is the minimal valuation of $\mathbf{\alpha}$. By the definition of the minimal valuation $s$, the expression $\sum A'_{ij}c'_{j, \mathbf{\alpha}}$ is not identically $0$. Now, since $F$ is algebraically closed, we can find many matrices in $\GL_{n}(\mathcal{O}_{F}/\mathfrak{m})$ missing the finitely many nontrivial linear conditions $\sum A'_{ij}c'_{j, \mathbf{\alpha}} = 0$ on their coefficients.

In the sequel, we will assume that, for each $k$, we have chosen $A$ such that for all $i = 1, \ldots, n$, the power series $$\sum_{j = 1}^{n}A_{ij}(\varphi_{j, k} - x_{j})$$ has the same minimal-polydegree terms as the collection $\varphi_{i, k} - x_{i}$, with each term having valuation bounded by $h(k)$. We also write $c_{i, \mathbf{\alpha}, k}$ for the $\mathbf{x}^{\mathbf{\alpha}}$-term of $\sum A_{ij}(\varphi_{j, k} - x_{j})$ and not for the $\mathbf{x}^{\mathbf{\alpha}}$-term of $\varphi_{i, k}$.

Now, suppose there is a root of polyvaluation $(m_{1}, \ldots, m_{n})$, and that one of its dominant terms in $\sum A_{ij}(\varphi_{j, k} - x_{j})$ is $\mathbf{x}^{\mathbf{\alpha}}$. We will bound the total degree $\alpha_{1} + \ldots + \alpha_{n}$. If $\mathbf{\alpha}$ is one of the lowest polydegrees of the collection $\varphi_{i, k} - x_{i}$ then its total degree is already bounded by $e$, so assume it is not minimal. For any other vector $\beta$, we have $c_{i, \mathbf{\beta}, k} - c_{i, \mathbf{\alpha}, k} \geq \mathbf{m}\cdot(\mathbf{\beta} - \mathbf{\alpha})$, with equality if and only if $c_{i, \mathbf{\beta}, k}\mathbf{x}^{\mathbf{\beta}}$ is another dominant term.

By the definition of minimal polydegree, there exists some minimal polydegree $\mathbf{\beta} \prec \mathbf{\alpha}$. For this polydegree, we have on the one hand $v(c_{i, \mathbf{\beta}, k}) \leq h(k)$, but on the other hand $c_{i, \mathbf{\beta}, k} - c_{i, \mathbf{\alpha}, k} \geq \mathbf{m}\cdot(\mathbf{\beta} - \mathbf{\alpha})$. We rearrange and get

\begin{align*}
\mathbf{m}\cdot(\mathbf{\beta} - \mathbf{\alpha}) &\leq h(k) - c_{i, \mathbf{\alpha}, k}\\
&\leq h(k) \quad \mbox{since $\varphi$ and its iterates have positive-valuation coefficients.}\\
\end{align*}
\noindent Conversely, $\mathbf{\beta} - \mathbf{\alpha}$ has nonnegative entries; hence, $$m(\beta_{1} + \ldots + \beta_{n} - \alpha_{1} - \ldots - \alpha_{n}) \leq \mathbf{m}\cdot(\mathbf{\beta} - \mathbf{\alpha}) \leq h(k).$$ This implies that the total degree of $\mathbf{\beta}$ is bounded by $e + h(k)/m$.

We now apply Theorem~\ref{rabinoff} and Corollary~\ref{volbound}. The polytopes in $\mathbb{R}^{n}$ arising as the convex hulls of the dominant terms in each $\sum A_{ij}(\varphi_{j, k} - x_{j})$ are in fact spanned by elements of total degree bounded by $e + h(k)/m$, which means they are contained in the simplex $(e + h(k)/m)U$ where $U$ is the unit simplex, spanned by the origin and by the unit vectors (or points) $\mathbf{e}_{1}, \ldots, \mathbf{e}_{n}$. The volume of $(e + h(k)/m)U$ has growth rate $(e + h(k)/m)^{n} \in O(k)$. This means that if we take $k$  and $m$ large enough, then we can ensure $(e + h(k)/m)^{n}\Vol(U) < k$. This means that the equations $\sum A_{ij}\varphi_{j, k} = \sum A_{ij}x_{i}$ cannot intersect in $k$ points of polyvaluation less than $m$. Since $A \in \GL_{n}(\mathcal{O}_{F})$, this means the equations $\varphi_{i, k} = x_{i}$ cannot intersect in $k$ points of polyvaluation less than $m$, either. Since we have picked $m$ to avoid any possible points of smaller period, this means that $\varphi_{i, k} = x_{i}$ cannot intersect in any point of polyvaluation less than $m$, or else there would be $k$ such points.\end{proof}

We are now ready to prove the main result of this section.

\begin{proof}[Proof of Theorem~\ref{isolatedindiff}.]We will find explicit formulas for the lowest-polydegree terms of the collection $\varphi_{i, k}(x_{1}, \ldots, x_{n}) - x_{i}$ for each $k$. We will show their minimal valuations grow logarithmically in $k$, and apply Lemma~\ref{implies}.

As in the proof of Theorem~\ref{rl}, we assume $\lambda_{i} \equiv 1 \mod \mathfrak{m}$ for each $i$, replacing $\varphi$ by an iterate if necessary; if because $F$ is so big we cannot make this assumption, and in fact $\lambda_{i}^{k} \not\equiv 1 \mod \mathfrak{m}$ for any $k$, then the theorem is in fact substantially easier to prove, if not so trivial as in the one-dimensional case. We also assume that if $\lambda_{i}$ is a root of unity, then $\lambda_{i} = 1$.


First, suppose $\lambda_{i}$ is not a root of unity. In that case, $\varphi_{i, k} - x_{i}$ has a nonzero $x_{i}$-term, equal to $\lambda_{i}^{k} - 1$. We write $\lambda_{i} = 1 + b_{1}\pi^{e_{1}} + \ldots$, replacing $\varphi$ by an iterate if necessary to ensure $v(e_{1}) > 1/(p-1)$, and as in the proof of Theorem~\ref{rl}, we obtain $v(\lambda_{i}^{k} - 1) \leq e_{1} + \log_{p}k$.

Using Lemma~\ref{implies}, the theorem follows immediately in case all multipliers $\lambda_{i}$ are not roots of unity, since in that case the minimal-polydegree terms of the collection $\varphi_{1, k} - x_{1}, \ldots, \varphi_{n, k} - x_{n}$ are the linear terms $x_{i}$, with coefficients of valuation that grows as $\log_{p}k + v(\lambda_{i} - 1)$.

So now, assume some of the multipliers are roots of unity. Replacing $\varphi$ with an iterate, we may assume all such multipliers are equal to $1$. Note that if $\lambda_{i}$ is not a root of unity, then $x_{i}$ is still a minimal-polydegree term in $\varphi_{i, k} - x_{i}$. Therefore, the only minimal-polydegree terms we need to consider in the sequel are of the form $x_{1}^{\alpha_{1}}\ldots x_{r}^{\alpha_{r}}$, where $\lambda_{1} = \ldots = \lambda_{r} = 1$.

Unlike in $1$ dimension, we may have nontrivial Jordan blocks. We compute entries of the associated matrix directly. If the Jordan block is of size $s$, we have $$\begin{bmatrix} 1 & 1 & \cdots & 0 \\ 0 & 1 & \cdots & 0 \\ \vdots & \vdots & \ddots & \vdots \\ 0 & 0 & \cdots & 1\end{bmatrix}^{k}=\begin{bmatrix} 1 & k & \cdots & {k \choose s-1} \\ 0 & 1 & \cdots & {k \choose s-2} \\ \vdots & \vdots & \ddots & \vdots \\  0 & 0 & \cdots & 1\end{bmatrix}.$$




Assume without loss of generality that this Jordan block is in fact $x_{1}, \ldots, x_{s}$. Then for $i = 1, \ldots, s-1$, we have a $kx_{i+1}$-term in $\varphi_{i, k} - x_{i}$. Thus, $x_{2}, \ldots, x_{s}$ are minimal-polydegree terms in the collection $\varphi_{i, k} - x_{i}$, with coefficient $k$, whose valuation is of course logarithmic in $k$, bounded by $\log_{p}k$.

Thus, in the sequel, it suffices to look just at monomials involving powers of variables $x_{i}$ for which not only $\lambda_{i} = 1$ but also $x_{i}$ is at the head of its Jordan block. Now, let $x_{1}^{\alpha_{1}}\ldots x_{n}^{\alpha_{n}}$ be a minimal-polydegree term of the collection $\varphi_{i} - x_{i}$, for which $\alpha_{j} > 0$ only if $\lambda_{j} = 1$ and $x_{j}$ is at the head of its Jordan block. We will show that it is also a minimal-polydegree term of the collection $\varphi_{i, k} - x_{i}$, with coefficient growing logarithmically in $k$. Set $c_{i, \mathbf{\alpha}, k}$ for its coefficient in $\varphi_{i, k}$.

Assume by induction that $\mathbf{x}^{\mathbf{\alpha}}$ is a minimal-polydegree term of the collection $\varphi_{i, k} - x_{i}$. We will show it is also a minimal-polydegree term for $\varphi_{i, k+1} - x_{i}$. First, we write $$\varphi_{i, k+1}(x_{1}, \ldots, x_{n}) = \varphi_{i, k}(\varphi_{1}, \ldots, \varphi_{n}) = \sum_{\mathbf{\beta} \in \mathbb{Z}^{n}}c_{i, \mathbf{\beta}, k}(\varphi_{1})^{\beta_{1}}\ldots(\varphi_{n})^{\beta_{n}}.$$ Observe that for each $\mathbf{\gamma} \in \mathbb{Z}^{n}$, the $\mathbf{x}^{\mathbf{\gamma}}$-term of $\varphi_{i, k+1}$ depends on finding $\beta_{1} + \ldots + \beta_{j}$ vectors summing to $\gamma$, which we label as $\mathbf{\gamma}^{j, 1}, \ldots, \mathbf{\gamma}^{j, \beta_{j}}$ for each $j = 1, \ldots, n$, and taking the $\mathbf{x}^{\mathbf{\gamma}^{j, j'}}$-term of $\varphi_{j}$. This means $$c_{i, \mathbf{\gamma}, k+1} = \sum_{\mathbf{\beta} \in \mathbb{Z}^{n}}c_{i, \mathbf{\beta}, k}\sum_{\sum\mathbf{\gamma}^{j, j'} = \mathbf{\gamma}}\prod_{j = 1}^{n}\prod_{j' = 1}^{\beta_{j}}c_{j, \mathbf{\gamma}^{j, j'}}.$$

Let us unpack what this means. If $\gamma \preceq \alpha$ then $c_{i, \mathbf{\gamma}, k+1}$ depends on coefficients of the form $c_{j, \mathbf{\gamma}^{j, j'}}$ where $\mathbf{\gamma}^{j, j'} \preceq \mathbf{\gamma}$. Now if $\mathbf{\gamma} \prec \mathbf{\alpha}$ then by the minimality of $\mathbf{\alpha}$ this means all terms $\mathbf{\gamma}^{j, j'}$ must be linear, of the form $\mathbf{e}_{j}$ where $x_{j}$ is at the head of its Jordan block. Moreover, the coefficient is $c_{j, \mathbf{e}_{j}} = \lambda_{j} = 1$. Since we need to have $\sum\mathbf{\gamma}^{j, j'} = \mathbf{\gamma}$, it follows that $\mathbf{\beta} = \mathbf{\gamma}$ and then $c_{i, \mathbf{\gamma}, k+1} = c_{i, \mathbf{\gamma}, k} = 0$ by the induction hypothesis.

Now, finally, if $\gamma = \alpha$, then we can set $\mathbf{\gamma}^{j, j'} = \mathbf{e}_{j}$ as before and obtain a contribution of $c_{i, \mathbf{\alpha}, k}$. But we can also set $\mathbf{\gamma}^{j, 1} = \mathbf{\alpha}$ for one value of $j$, for which $\beta_{j} = 1$, and $\mathbf{\gamma}^{j', 1} = 0$ for $j' \neq j$, which forces $\beta_{j'} = 0$ since $\varphi_{1}, \ldots, \varphi_{n}$ have zero constant terms. In other words, we can set $\mathbf{\beta} = \mathbf{e}_{j}$ and then obtain a contribution equal to $c_{i, \mathbf{e}_{j}, k}c_{j, \mathbf{\alpha}}$. Taking everything together, we get $$c_{i, \mathbf{\alpha}, k+1} = c_{i, \mathbf{\alpha}, k} + \sum_{j = 1}^{n}c_{i, \mathbf{e}_{j}, k}c_{j, \mathbf{\alpha}}.$$

Observe that if all Jordan blocks are simple, then $c_{i, \mathbf{e}_{j}, k} = \lambda_{i}^{k}\delta_{ij}$ and then we obtain $c_{i, \mathbf{\alpha}, k+1} = c_{i, \mathbf{\alpha}, k} + \lambda_{i}^{k}c_{i, \mathbf{\alpha}}$. However, we can also get the same formula if we choose $i$ carefully enough. Specifically, we choose $i$ to be such that $c_{i, \mathbf{\alpha}} \neq 0$, and for each $j > i$ in the same Jordan block, $c_{j, \mathbf{\alpha}} = 0$; such an $i$ is guaranteed to exist, since $\mathbf{\alpha}$ is a minimal-polydegree term, i.e. there does exist some $\varphi_{i}$ where it has nonzero coefficient, and then we can pick the maximal such $i$. Then if $c_{i, \mathbf{e}_{j}, k} \neq 0$ then either $c_{j, \mathbf{\alpha}} = 0$ or $j = i$. For this $i$, we obtain $c_{i, \mathbf{\alpha}, k+1} = c_{i, \mathbf{\alpha}, k} + \lambda_{i}^{k}c_{i, \mathbf{\alpha}}$.

To finish the proof of the theorem, we need to prove that the recursive formula $c_{i, \mathbf{\alpha}, k+1} = c_{i, \mathbf{\alpha}, k} + \lambda_{i}^{k}c_{i, \mathbf{\alpha}}$ defines a sequence $c_{i, \mathbf{\alpha}, k}$ that is nonzero for every $k$ and has valuation growing logarithmically in $k$. We replace the recursive formula with the closed-form formula $c_{i, \mathbf{\alpha}, k} = c_{i, \mathbf{\alpha}}(1 + \lambda_{i} + \ldots + \lambda_{i}^{k-1})$; this can be seen by writing $a_{k} = c_{i, \mathbf{\alpha}, k}$ and then expanding $a_{k+1} = a_{k} + \lambda_{i}^{k}a_{1}$ as $$a_{k+1} = a_{k} + \lambda_{i}^{k}a_{1} = a_{k-1} + \lambda_{i}^{k-1}a_{1} + \lambda_{i}^{k}a_{1} = \ldots = a_{1}(1 + \ldots + \lambda_{i}^{k}).$$ Now, we split into two cases, depending on whether $\lambda_{i} = 1$ or not.

If $\lambda_{i} = 1$, then $c_{i, \mathbf{\alpha}, k+1} = c_{i, \mathbf{\alpha}, k} + c_{i, \mathbf{\alpha}}$; this is an arithmetic sequence, with closed-form formula $c_{i, \mathbf{\alpha}, k} = kc_{i, \mathbf{\alpha}}$. Clearly, $\mathbf{x}^{\alpha}$ remains a minimal-polydegree term, and its coefficient's valuation grows as $v(k) + v(c_{i, \mathbf{\alpha}})$. Observe that this is the same formula we obtain in the proof of Theorem~\ref{rl}, in case $\chr F = 0$ and $\lambda = 1$: we get $v(k) + v(a_{e})$, where $a_{e}$ is the minimal-degree term of $\varphi(z) - z$.

Finally, if $\lambda_{i} \neq 1$, then we get $$c_{i, \mathbf{\alpha}, k} = c_{i, \mathbf{\alpha}}(1 + \lambda_{i} + \ldots + \lambda_{i}^{k-1}) = c_{i, \mathbf{\alpha}}\frac{\lambda_{i}^{k} - 1}{\lambda_{i} - 1}.$$ But now, as in the proof of Theorem~\ref{rl} when $\chr F = 0$ and $\lambda$ is not a root of unity, we have $v(\lambda_{i}^{k} - 1) = v(k) + v(\lambda_{i} - 1)$ as long as $v(\lambda_{i} - 1) \geq 1/(p-1)$; replacing $\varphi$ with an iterate if necessary, we can in fact guarantee $v(\lambda_{i} - 1) \geq 1/(p-1)$. Thus $c_{i, \mathbf{\alpha}, k} = v(k) + v(c_{i, \mathbf{\alpha}})$ again. This is analogous to the $1$-dimensional formula when $\lambda = 1$ rather than when $\lambda$ is not a root of unity, since we are ultimately assuming that we have some multipliers $\lambda_{j}$ that are roots of unity, and computing the valuation of the lowest-polydegree nonlinear terms rather than the valuation of the linear term minus $1$.

We have proven that for whenever $\mathbf{\alpha}$ is a minimal polydegree for the family $\varphi_{1} - x_{1}, \ldots, \varphi_{n} - x_{n}$, it is also a minimal polydegree for $\varphi_{1, k} - x_{1}, \ldots, \varphi_{n, k} - x_{n}$, and moreover, there exists $i$ such that the coefficient $c_{i, \mathbf{\alpha}, k}$ has valuation growing logarithmically in $k$. Thus we can take $t$ to be the maximum valuation of each the minimal polydegrees of the family $\varphi_{i} - x_{i}$, and then $h(k) = \log_{p}k + t$ will satisfy the conditions of Lemma~\ref{implies}.\end{proof}

If $\chr F > 0$ and all multipliers of $\varphi$ are irrationally indifferent, then we cannot use the above technique to prove Theorem~\ref{isolatedindiff} the way we proved Theorem~\ref{rl}. The reason is that although the lowest-polydegree terms of the system $\varphi_{i, k} - x_{i}$ are the linear ones, their coefficients have valuations that grow linearly in $k$, as in the proof of Theorem~\ref{rl}; say the explicit bound is $v(\lambda_{i}^{k} - 1) \leq ck$. Thus, if we attempt to bound the polyvaluation of period-$k$ points by $m$, we find that we can only force the dominant terms in each $\varphi_{i, k}$ to lie inside a simplex of side $ck/m + 1$; the volume of this simplex grows on the order of $k^{n}$ rather than $k$, and therefore we cannot pick $m$ such that the volume is eventually smaller than $k$.

Nonetheless, we conjecture that,

\begin{conj}\label{isolatedconj}Let $\varphi$ be an analytic map, defined over $F$ of characteristic $p$, from the open unit polydisk in $n$ dimensions to itself. If $x = (0, \ldots, 0)$ is a fixed point all of whose multipliers are indifferent, then either $x$ is isolated as a periodic point of $\varphi$, or, for some iterate $\varphi^{k}$, there exists a pointwise fixed subvariety of positive dimension passing through $x$.\end{conj}

\begin{rem}Using the proof of Theorem~\ref{isolatedper} assuming Theorem~\ref{isolatedindiff}, it is easy to show that if Conjecture~\ref{isolatedconj} is true, then in fact every fixed point with nonrepelling multipliers is isolated.\end{rem}

The reason we make this conjecture is that, in characteristic $p$, nonlinear low-degree terms of $\varphi^{k}$ vanish after taking $p$th iterates. For example, in one dimension, if we write $\varphi(z) = \lambda z + z^{e} + \ldots$ where $\lambda$ is not a root of unity, then $$\varphi^{k}(z) = \lambda^{k}z + \lambda^{k-1}(1 + \lambda^{e-1} + \ldots + \lambda^{(e-1)(k-1)})z^{e} + \ldots$$ and then the $z^{e}$-term vanishes whenever $p \mid k$. If enough low-degree terms vanish -- in case the multipliers are all irrationally indifferent, it suffices for the terms below total degree $k$ to vanish when $k$ is a prime power -- then we can bound the valuations of the periodic points. In ongoing work of the author with Lindahl, this method is used to prove that in some additional $1$-dimensional cases not covered by Lindahl and Rivera-Letelier in~\cite{LRL2}, rationally indifferent points in characteristic $p$ are isolated.

\section{Some Cases of Zhang's Conjecture}\label{zhang}

In~\cite{ABR}, Amerik-Bogomolov-Rovinsky prove Conjecture~\ref{introzhang} in case the multipliers are multiplicatively independent. Specifically, they show, in Proposition 2.3 and Corollaries 2.7 and 2.8:

\begin{thm}\label{abr}Let $\varphi = (\varphi_{1}, \ldots, \varphi_{n})$ be an analytic map on the closed unit polydisk, defined over a complete nonarchimedean field $F$ of characteristic $0$. Suppose that the point $x = (0, \ldots, 0)$ is fixed and has nonrepelling, multiplicatively independent multipliers, all of which are algebraic numbers. Then $\varphi$ is linearizable at $x$; moreover, after linearization, if $z = (z_{1}, \ldots, z_{n}) \in \mathcal{O}_{F}^{n}$ misses all the hyperplanes $x_{i} = 0$, then the orbit of $z$ is never contained in any power series $g \in F[[x_{1}, \ldots, x_{n}]]$ that converges on $\mathcal{O}_{F}^{n}$. In particular, if $\varphi$ is equivalent to a morphism $\psi$ with coefficients in a number field $K$ with $F = K_{v}$, $z$ is the image of an algebraic point under the linearization map, and $g$ is the image of a polynomial in $K[x_{1}, \ldots, x_{n}]$, then the orbit of $z$ is not contained in $g$.\end{thm}

\begin{rem}The assumption that the multipliers of $x$ are nonrepelling is important in the analytic proof, but once we look at multipliers on the number field $K$ rather than on the local field $F$, the assumption is no longer required, since we can pick $F = K_{v}$ to miss all denominators of the multipliers of $x$. In fact, if we pick $v$ to be a prime of good reduction, then the multipliers of any periodic point are never repelling; see~\cite{Hsia96}.\end{rem}

Using the tools developed in Section~\ref{separating}, we can extend Theorem~\ref{abr} to cover some cases in which the multipliers are not multiplicatively independent.

\begin{thm}\label{zero}Let $\varphi: \mathbb{P}^{n} \to \mathbb{P}^{n}$ be a morphism defined over a number field $K$. Suppose that there exists a fixed point $x$ whose multipliers consist of a single $0$ and $n-1$ multiplicatively independent numbers, and that $K$ is large enough that $x \in \mathbb{P}^{n}(K)$ and the rational canonical form of $\varphi_{*}(T_{x})$ is a diagonal matrix. Then there exists a point $z \in \mathbb{P}^{n}(K)$ whose forward orbit under $\varphi$ is Zariski-dense.\end{thm}

\begin{proof}We choose a valuation $v$ of $K$ such that the nonzero multipliers are all indifferent, and set $F = K_{v}$. We dehomogenize and change coordinates so that $x = (0, \ldots, 0)$ as in the case of Lemma~\ref{formsep}. By Lemma~\ref{anal1}, we can do an analytic coordinate-change such that $x_{1} = 0$ is a fixed subvariety, which we call $V$; here we choose coordinates such that $\lambda_{1} = 0$. At no point in the proof of Lemma~\ref{formsep} do we enlarge our field of definition; once $x \in F$ and $\lambda_{i} \in F$ for all $i$, we are only performing field operations. Thus, it is no obstacle that $F$ is not algebraically closed.

We have $\varphi_{1} = x_{1}g_{1}$ where $g_{1} \in \mathcal{O}_{F}[[x_{1}, \ldots, x_{n}]]$ has zero constant. We can locally linearize $\varphi|_{V}$ by Theorem~\ref{abr}, and this means that for $i > 1$, we have $\varphi_{i} = \lambda_{i}x_{i} + x_{1}g_{i}$, where $g_{i} \in \mathcal{O}_{F}[[x_{1}, \ldots, x_{n}]]$.

Now, suppose $z = (z_{1}, \ldots, z_{n})$, such that $|z_{1}| < |z_{i}| < 1$ for all $i$. The dominant term in $\varphi_{i}$ for $i > 1$ is always $\lambda_{i}x_{i}$, while $|\varphi_{1}(z_{1}, \ldots, z_{n})| < |z_{1}|$; thus the dominant term in $\varphi_{i, k}$ for $i > 1$ is $\lambda_{i}^{k}x_{i}$ as well. By Proposition~\ref{rate}, the orbit of $z$ is attracted to $V$. Suppose also that the orbit of $z$ is contained in a closed analytic subvariety $W$, fixed under $\varphi$, arising as the zero set of convergent power series.

Suppose first that $W$ is not $V$ (which means $W$ cannot contain $V$, since then it would just be the entire space). Then $W$ must intersect $V$. This is because $F$ is a finite extension of a $p$-adic field, and thus its closed unit polydisk is compact, making $W$ and $V$ compact as well, and now we use the fact that $W$ contains a set whose distance to $V$ approaches $0$ to show that $W$ must intersect $V$. Moreover, $V \cap W$ is closed and fixed under $\varphi$, and arises as the zero set of convergent power series. Let $z' \in V \cap W$; its orbit is contained in $V \cap W$, and in particular it is contained in the zero set of a convergent power series on $F$, e.g. one of the series defining $W$. By Theorem~\ref{abr}, this implies $z'$ lies on some hyperplane $x_{i} = 0$ with $i > 1$. But now we have $|\varphi_{i, k}(z_{1}, \ldots, z_{n})| = |z_{i}|$ whenever $|z_{1}| < |z_{j}| < 1$ for all $j$. Thus if we pick $z$ to be such that $z_{i} \neq 0$ for all $i$, the subvariety $W$ stays bounded away from any hyperplane $x_{i} = 0$, and therefore $z'$ cannot lie on any hyperplane of the form $x_{i} = 0$, giving us a contradiction.

So now we are reduced to the case when $W = V$. To rule it out, we will find $z$ whose forward image approaches $V$ but does not actually land on it. This means finding a suitable region in which $g_{1} \neq 0$. Suppose first that $g_{1} \in F[[x_{1}]]$; then, possibly after conjugating by a linear map $x_{1} \mapsto cx_{1}$, we get $\varphi_{1} = x_{1}^{e} + \ldots$ for some $e > 1$, and then for $0 < |z_{1}| < 1$ the orbit never lands on zero. Now suppose that $g_{1}$ has some monomials involving terms other than $x_{1}$. If $|z_{1}|$ is sufficiently small relative to $|z_{2}|, \ldots, |z_{n}|$, then the dominant terms of $g_{1}$ will be those with the lowest $x_{1}$-degree. Among the terms with minimal $x_{1}$-degree, there are finitely many of minimal $(x_{2}, \ldots, x_{n})$-polydegree, and we can find many values of $(|z_{2}|, \ldots, |z_{n}|) \in \mathbb{Z}^{n-1}$ such that exactly one of those terms will dominate, preventing $g_{1}$ from being zero. Moreover, since $|\varphi_{i}(z_{1}, \ldots, z_{n})| = |z_{i}|$, this dominant term will dominate in the entire orbit of $z$, preventing the orbit from landing on $V$.

So far, we have produced an open region in the open unit polydisk on which the orbit of $z$ either lands on $V$ or is not contained in any fixed closed analytic subvariety; moreover, we have produced many assignments of valuations of $z_{2}, \ldots, z_{n}$ for which the orbit of $z$ does not land on $V$. In $K_{v}$ this is an open set. Moreover, if we invert our linearization and our coordinate change from Lemma~\ref{formsep}, this region remains open and still satisfies the same properties. Since $K$ is dense in $K_{v}$, this means there are many points in this region defined over $K$. Finally, if a point has an orbit under the algebraic map $\varphi$ that is not Zariski-dense, then the orbit is contained in a closed algebraic subvariety, defined as the zero set of some polynomials in $x_{1}, \ldots, x_{n}$, which are mapped to convergent power series under local analytic coordinate changes; thus, in this region, every point has a Zariski-dense orbit.\end{proof}

\begin{thm}\label{hack}Let $\varphi: \mathbb{P}^{2} \to \mathbb{P}^{2}$ be a morphism defined over a number field $K$. Suppose that there exists a fixed point $x$ defined over $K$ whose multipliers consist of $1$ and a number that is not a root of unity. Then there exists a point $z \in \mathbb{P}^{n}(K)$ whose forward orbit under $\varphi$ is Zariski-dense.\end{thm}

\begin{proof}We fix coordinates so that $x = (0, 0)$, $\lambda_{2} = 1$ (replacing $\varphi$ with an iterate if necessary) and $\lambda_{1}$ is not a root of unity. We also choose $F = K_{v}$ (extending to $\mathbb{C}$ if $K_{v} = \mathbb{R}$) to be such that $|\lambda_{1}| < 1$; since $\lambda_{1}$ is not a root of unity, such $F$ is guaranteed to exist. We apply Lemmas~\ref{formsep} and~\ref{anal1} twice, once to analytically conjugate so that $x_{1} \mid \varphi_{1}$ and the second time to analytically conjugate so that $x_{2} \mid \varphi_{2}$. We check that, if we have already fixed an analytic conjugation such that $x_{1} \mid \varphi_{1}$, conjugating by $x_{2} \mapsto x_{2} - f_{2}$ is not going to change this relation, and thus $x_{i} \mid \varphi_{i}$ can hold simultaneously for $i = 1, 2$.

If $\lambda_{1} = 0$, then we are essentially in the same situation as in Theorem~\ref{zero}. We choose $v$ to be any nonarchimedean valuation on $K$, and apply the proof of Theorem~\ref{zero} verbatim, except that $\varphi_{2}$ has some higher-order terms in $x_{2}$; however, if $|z_{1}| < |z_{2}| < 1$ then the $x_{2}$ term will still dominate at $(z_{1}, z_{2})$, as required. Finally, instead of applying Theorem~\ref{abr} to rule out fixed analytic subvarieties near fixed points with multiplicatively independent multipliers, we apply Theorem~\ref{rl} to rule out an accumulation of periodic points near $x$. In the sequel, we can then assume $\lambda_{1} \neq 0$.

By standard results of normal forms (see e.g.~\cite{D1C}), we can analytically conjugate $x_{1}$ and $x_{2}$ separately, so that the restriction of $\varphi$ to $x_{2} = 0$ is linearized to $\lambda_{1}x_{1}$ and the restriction to $x_{1} = 0$ is $x_{2} + x_{2}^{e}$, where $e > 1$ is the order of $x$ as a root of the fixed point power series $\varphi(0, x_{2}) - x_{2}$. We thus have $\varphi_{1} = \lambda_{1}x_{1} + x_{1}x_{2}g_{1}$ and $\varphi_{2} = x_{2} + x_{2}^{e} + x_{1}x_{2}g_{2}$.

If $v$ is nonarchimedean, then we apply the same method as in the proof of Theorem~\ref{zero}. If $|z_{1}| < |z_{2}| < |\lambda_{1}|$ then the dominant term of $\varphi_{2}$ at $(z_{1}, z_{2})$ is $x_{2}$ and then $|\varphi_{2}(z_{1}, z_{2})| = |z_{2}|$, and the dominant term of $\varphi_{1}$ is $\lambda_{1}x_{1}$ and then $|\varphi_{1}(z_{1}, z_{2})| = |\lambda_{1}||z_{1}|$. We are then in the same situation as in Theorem~\ref{zero}: any fixed analytic subvariety $W$ will intersect the line $x_{1} = 0$, and because periodic points near $0$ are isolated, we get a contradiction unless $W$ actually is the line $x_{1} = 0$. It is in fact easier to ensure $W$ is not the line $x_{1} = 0$ than in the proof Theorem~\ref{zero}, since the preimage of $x_{1} = 0$ under $\varphi$ consists now of $x_{1} = 0$ and $\lambda_{1} + x_{2}g_{1} = 0$, and when $|z_{1}|, |z_{2}|$ are sufficiently small, $\lambda_{1}$ is the dominant term and thus the equation is never satisfied.

So now we assume $v$ is archimedean, and $F = \mathbb{C}$. We can still talk about dominant terms, and we can still apply Proposition~\ref{rate} and argue that any fixed $V$ intersects $x_{1} = 0$, but the dynamics on the line $x_{1} = 0$ is not so simple. In particular, we cannot ensure $|x_{2}|$ stays the same on the orbit of $z_{2}$.

We will simplify the action of $\varphi$ on $(x_{1}, x_{2})$ to make it partly a coordinatewise map. We cannot use coordinate changes of the form $x_{2} \mapsto x_{2} + f(x_{1}) + g(x_{2})$, as those would run afoul of the uniqueness of both the linearization (which gives a unique $g$) and the coordinate change in Lemma~\ref{formsep} (which gives a unique $f$). However, we can change coordinates by $x_{2} \mapsto x_{2}(1 + \alpha_{d_{1}, d_{2}}x_{1}^{d_{1}}x_{2}^{d_{2}})$. This means replacing $x_{2}$ with $$x_{2}\left(1 - \alpha_{d_{1}, d_{2}}x_{1}^{d_{1}}x_{2}^{d_{2}} + \frac{\alpha_{d_{1}, d_{2}}^{2}}{d_{2}+1}x_{1}^{2d_{1}}x_{2}^{2d_{2}} - \ldots\right)$$ and then replacing $\varphi_{2}$ with $$\varphi_{2}\cdot(1 + \alpha_{d_{1}, d_{2}}(\varphi_{1})^{d_{1}}(\varphi_{2})^{d_{2}}).$$ This does not change any terms in $\varphi_{2}$ below $x_{1}^{d_{1}}x_{2}^{d_{2}+1}$, which is incremented by $(\lambda_{1}^{d_{1}} - 1)\alpha_{d_{1}, d_{2}}$. We can thus choose a sequence of $\alpha_{d_{1}, d_{2}}$s that kills all of $g_{2}$, at least formally.

We will now show that the sequence $\alpha_{d_{1}, d_{2}}$ defines a local analytic coordinate change. For this, we need to show that $$x_{2}\prod_{d_{1}, d_{2} = (1, 0)}^{(\infty, \infty)}(1 + \alpha_{d_{1}, d_{2}}x_{1}^{d_{1}}x_{2}^{d_{2}})$$ converges and is locally invertible in some polydisk around $(0, 0)$. The convergence part is equivalent to showing that $$\sum_{d_{1}, d_{2} = (1, 0)}^{(\infty, \infty)}\alpha_{d_{1}, d_{2}}x_{1}^{d_{1}}x_{2}^{d_{2}}$$ converges. Observe that if we can bound $\alpha_{d_{1}, d_{2}}$ by some exponential $c^{d_{1} + d_{2}}$ where $c$ is positive real then we easily get convergence when $|x_{1}|, |x_{2}| < 1/c$. We also have $1 + \alpha_{d_{1}, d_{2}}x_{1}^{d_{1}}x_{2}^{d_{2}} \neq 0$, so the infinite product is nonzero; observe that the terms of the inverse grow more slowly than those of the reciprocal, so it suffices to show the product is nonzero. Now, since $1-\lambda_{1}^{d_{1}}$ is bounded away from $0$, all we need to prove is that, as we change coordinates by $\alpha_{d_{1}, d_{2}}$ for higher and higher $d_{1}$ and $d_{2}$, the remaining $x_{1}^{d'_{1}}x_{2}^{d'_{2}}$-terms stay bounded by a uniform $c^{d'_{1} + d'_{2}}$; this will prove both that the infinite product converges and that it is nonzero in a polydisk of positive radius $1/c$.

Before we do any coordinate change, all terms of $\varphi_{1}$ and $\varphi_{2}$ are by assumption bounded by $1$. Each time we do a coordinate change by $\alpha_{d_{1}, d_{2}}$, we divide by $1-\lambda_{1}^{d_{1}}$ and add a number of terms. If we choose to change coordinates in order of increasing $d_{1} + d_{2}$ and then in order of increasing $d_{2}$, then the number of terms we add to $x_{1}^{d_{1}}x_{2}^{d_{2}}$ in prior coordinate changes is (weakly) bounded by the number of prior changes, which grows polynomially in $d_{1} + d_{2}$, times the number of terms added per coordinate change. The number of lower terms grows polynomially as well.

Therefore, we just need to check what happens when we take a single monomial and change it by $x_{2} \mapsto x_{2}(1 + \alpha_{d_{1}, d_{2}}x_{1}^{d_{1}}x_{2}^{d_{2}})$. In the worst case, it will multiply the $x_{1}^{d'_{1}}x_{2}^{d'_{2}}$-term by binomial coefficients ${d'_{1} \choose \lfloor d'_{1}/d_{1}\rfloor}{d'_{2} \choose \lfloor d'_{2}/d_{2}\rfloor}$, which are exponential in $d'_{1} + d'_{2}$. To see that this is the case, observe that $${d'_{1} \choose \lfloor d'_{1}/d_{1}\rfloor} \leq {d'_{1} \choose \lfloor d'_{1}/2\rfloor} = \frac{d'_{1}!}{\lfloor d'_{1}/2\rfloor!\lceil d'_{1}/2\rceil!}$$ and now for large integers $n$, $n!$ grows on the order of $(n/e)^{n}$, where here (and nowhere else in this paper) $e = 2.718\ldots$ rather than an integer index. So now we obtain $$\frac{d'_{1}!}{\lfloor d'_{1}/2\rfloor!\lceil d'_{1}/2\rceil!} \sim \frac{(d'_{1})^{d'_{1}}}{\lfloor d'_{1}/2\rfloor^{d'_{1}}} \sim 2^{d'_{1}},$$ where the relation $A(k) \sim B(k)$ means $A(k) \in O(B(k))$ and $B(k) \in O(A(k))$. Now we need a growth rate that's at most exponential in $d'_{1} + d'_{2}$, so the $2^{d'_{1} + d'_{2}}$ we have found is enough to ensure the coordinate change is analytic near $(0, 0)$ and not just formal.

In the sequel, we will then assume $\varphi_{2} = x_{2} + x_{2}^{e}$. Note that we cannot similarly turn $\varphi_{1}$ into a function of $x_{1}$ alone, since we'd need to divide by $1 - \lambda_{2}^{d_{2}}$, which we can't since $\lambda_{2} = 1$.

So now, we start with a point $(z_{1}, z_{2})$, with $|z_{1}|$ and $|z_{2}|$ both small. Since $\varphi_{2}$ depends only on $z_{2}$, we will choose $z_{2}$ to lie in the Fatou set of $\varphi_{2}$. Recall that $\varphi_{2}$ has $e-1$ attracting petals arranged around $0$, with repelling directions in between. The repelling directions are those at which $z$ and $z^{e}$ have the same complex argument, and the attracting petals will attract along directions at which $z$ and $z^{e}$ have diametrically opposed complex arguments; for example, if $e = 2$, then $0$ is repelling along the positive real direction, and in any other direction points wrap around to the negative real line, along which $0$ is attracting. The rate of attraction is quite slow: in the limit, $|\varphi_{2, k}|$ shrinks on the order of $\sqrt[e-1]{1/k}$.

Proposition~\ref{rate} as stated applies only in the nonarchimedean case, but in this case we can extend it to $\mathbb{C}$. Within its region of convergence, $\varphi_{1}$ converges absolutely. Thus we can easily take $|z_{1}|, |z_{2}|$ sufficiently small so that $|z_{2}g_{1}(z_{1}, z_{2})| < |\lambda_{1}|/2$. In this region, $\varphi_{1, k}$ will approach $0$, and in fact will approach $0$ by a geometric rate, on the order of $1/\lambda_{1}^{k}$, in the limit.

Observe that since we are taking $z_{2}$ to lie in an attracting petal, any fixed curve $W$ containing the orbit of $z$ intersects $x_{1} = 0$ at the origin. Thus, we cannot apply any results on isolated periodic points -- indeed, over $\mathbb{C}$, rationally indifferent points are not isolated at all, as they lie in the Julia set (but their attracting petals contain no periodic points). Instead, we use a different argument: we will show that $W$ cannot possibly arise as a local analytic deformation of an algebraic curve.

An algebraic curve is defined by a polynomial equation in $x_{1}$ and $x_{2}$. The local analytic coordinate changes we have done all replace $x_{1}$ and $x_{2}$ with power series in $F[[x_{1}, x_{2}]]$, and not Laurent series. Thus if $W$ is an analytic deformation of an algebraic curve, it is defined by a power series. We will derive a contradiction.

Suppose $W$ is the zero set of the power series $g(x_{1}, x_{2})$, where we choose $g$ to be its own radical. Replacing $\varphi$ by an iterate if possible, we may assume $W$ is irreducible, so $g$ is irreducible. In particular, $g$ cannot have a unique term of minimal polydegree (Definition~\ref{minpoly}): if this minimal-polydegree term is nonconstant, say $x_{1}^{d_{1}}x_{2}^{d_{2}}$, then $g$ is divisible by $x_{1}^{d_{1}}x_{2}^{d_{2}}$, whereas if it is constant, then $g(0, 0) \neq 0$, which contradicts the fact that $(0, 0)$ lies on $W$.

Now, replace $g(x_{1}, x_{2})$ with $h(x_{1}, x_{2}) = g(\lambda_{1}x_{1} + x_{1}x_{2}g_{1}, x_{2} + x_{2}^{e})$. We write $g_{d_{1}, d_{2}}$ for the $x_{1}^{d_{1}}x_{2}^{d_{2}}$-term of $g$ and $h_{d_{1}, d_{2}}$ for the $x_{1}^{d_{1}}x_{2}^{d_{2}}$-term of $h$. If $x_{1}^{d_{1}}x_{2}^{d_{2}}$ has minimal polydegree in $g$, then we obtain $h_{d_{1}, d_{2}} = \lambda_{1}^{d_{1}}g_{d_{1}, d_{2}}$ and this is a minimal-polydegree term in $h$. Moreover, we do not get any new terms of minimal polydegree, since every term of $\varphi_{i}$ is divisible by $x_{i}$.

Since $W$ is fixed under $\varphi$, we have $g \mid h$. This is not a general fact of power series; it's a consequence of the fact that $g$ and $h$ are both local-analytic deformations of polynomials in $\mathbb{C}[x_{1}, x_{2}]$ by the same conjugation, whence $h/g$ is also a deformation of a polynomial. Since $g$ and $h$ have the same minimal polydegrees, $h/g$ must have a nonzero constant term, which we write as $c$. But now this implies that, if $(d_{1}, d_{2})$ is a minimal polydegree of $g$ and $h$, then $h_{d_{1}, d_{2}} = cg_{d_{1}, d_{2}}$, whence $c = \lambda_{1}^{d_{1}}$. This forces all minimal-polydegree terms to have the same $x_{1}$-degree, which must be $\log c/\log\lambda_{1}$. But two distinct minimal-polydegree terms cannot have the same $x_{1}$-degree, since then the one with the higher $x_{2}$-degree would not be minimal. This gives us the contradiction we need.\end{proof}

\begin{rem}Two parts of the proof of Theorem~\ref{hack} can be generalized to higher dimension, with one multiplier equal to $1$ and the rest multiplicatively independent: namely, if $\lambda_{n} = 1$, then it is possible to analytically change coordinates so that $\varphi_{n} = x_{n} + x_{n}^{e}$, and if $g \in F[[x_{1}, \ldots, x_{n}]]$ is irreducible and has zero constant then it cannot divide $g(\varphi)$.

We can in fact show more: if the ideal $(g_{1}, \ldots, g_{r})$ is prime and every $g_{i}$ has zero constant, then the power series $g_{1}, \ldots, g_{r}$ collectively have more than $r$ minimal polydegrees, or else we can replace the generators by an invertible linear combination and obtain some reducible $g_{i}$; on these at least $r+1$ minimal polydegrees, $\varphi$ acts with at least $r+1$ different weights, and this implies that we cannot write every $g_{i}(\varphi)$ as a linear combination of $g_{1}, \ldots, g_{r}$. This proves that not only are analytic hypersurfaces through $x$ never periodic, but also nonlinear analytic subvarieties in every dimension are never periodic. Unfortunately, this is not enough to generalize Theorem~\ref{hack}, as we cannot rule out analytic subvarieties that pass near $x$ but not through it; the trick of using divisibility in $F[[x_{1}, \ldots, x_{n}]]$ fails completely for maps with nonzero constants. Observe that even in the case of Theorem~\ref{hack}, our analytic subvarieties do not pass through $x$ if $F$ is nonarchimedean; and if there are several multipliers that are not roots of unity, then it is not guaranteed there is any $F$ such that these multipliers are all attracting.\end{rem}

\bibliographystyle{amsplain}
\bibliography{rigidity}

\end{document}